\documentclass[11pt]{amsart}
\usepackage{amsmath}
\usepackage{amssymb}
\usepackage{amsfonts}
\usepackage{amsthm}
\usepackage{enumerate}
\usepackage[mathscr]{eucal}
\usepackage{graphicx}
\usepackage[bookmarksnumbered, bookmarksopen, colorlinks, citecolor=blue, linkcolor=blue]{hyperref}

\usepackage{eufrak}
\usepackage{color}
\usepackage{tikz-cd}
\usepackage{enumerate}
\usepackage{cleveref}

\newtheorem{theorem}{Theorem}[section]
\newtheorem{lemma}[theorem]{Lemma}

\theoremstyle{definition}
\newtheorem{definition}[theorem]{Definition}

\theoremstyle{remark}
\newtheorem{remark}[theorem]{Remark}

\numberwithin{equation}{section}

\def\intslash{\rlap{\kern  .32em $\mspace {.5mu}\backslash$ }\int}
\def\qsl{{\rlap{\kern  .32em $\mspace {.5mu}\backslash$ }\int_{Q_x}}}
\def\Re{\operatorname{Re\,}}
\def\Im{\operatorname{Im\,}}

\def\M{{\mathcal M}}

\def\norm#1{{ \left|  #1 \right| }}

\def\set#1{{ \left\{ #1 \right\} }}

\def\Bone{{\textbf{1}}}

\newcommand{\Be}{\begin{equation}}
\newcommand{\Ee}{\end{equation}}
\newcommand{\Bes}{\begin{equation*}}
\newcommand{\Ees}{\end{equation*}}
\newcommand{\Bsp}{\begin{split}}
\newcommand{\Esp}{\end{split}}
\newcommand{\Bm}{\begin{multline}}
\newcommand{\Em}{\end{multline}}
\newcommand{\Bea}{\begin{eqnarray}}
\newcommand{\Eea}{\end{eqnarray}}
\newcommand{\Beas}{\begin{eqnarray*}}
\newcommand{\Eeas}{\end{eqnarray*}}
\newcommand{\Benu}{\begin{enumerate}}
\newcommand{\Eenu}{\end{enumerate}}
\newcommand{\Bi}{\begin{itemize}}
\newcommand{\Ei}{\end{itemize}}

\newcommand{\supp}{\operatorname{supp}}

\newcommand{\mN}{\operatorname{\mathcal{N}}}
\newcommand{\mM}{\operatorname{\mathcal{M}}}

\newcommand{\mS}{\operatorname{\mathcal{S}}}
\newcommand{\mQ}{\operatorname{\mathcal{Q}}}
\newcommand{\mP}{\operatorname{\mathcal{P}}}

\newcommand{\R}{\operatorname{\mathbb{R}}}
\newcommand{\C}{\operatorname{\mathbb{C}}}
\newcommand{\Z}{\operatorname{\mathbb{Z}}}
\newcommand{\N}{\operatorname{\mathbb{N}}}

\newcommand*{\dif}{{\rm d}}

\def\set #1{ \left\{ #1 \right\} }
\def\norm #1{ \left\| #1 \right\| }
\def\lbracket #1{ \left( #1 \right) }

\ifdefined\acknowledgments
    \newenvironment{acknowledgements}{\begin{acknowledgments}}{\end{acknowledgments}}
\else
    \newenvironment{acknowledgements}{\section*{Acknowledgements}}{\par}
\fi

\begin{document}

\title[Noncommutative maximal operators]{Noncommutative weak type $(1,1)$ estimates for Calder\'on-Zygmund operators with a class of $L_1$-integral conditions}

\author{Xudong Lai}
\address{Xudong Lai: Institute for Advanced Study in Mathematics\\
Harbin Institute of Technology\\
Harbin
150001\\
China;
Zhengzhou Research Institute\\
Harbin Institute of Technology\\
Zhengzhou
450000\\
China}
\email{xudonglai@hit.edu.cn}

\author{Lingxin Xu}
\address{Lingxin Xu: Institute for Advanced Study in Mathematics, Harbin Institute of Technology, Harbin, 150001, China}
\email{xulingxin97@163.com}

\thanks{This work is supported by National Natural Science Foundation of China (No. 12322107, No. 12271124, and No. W2441002) and Heilongjiang Provincial Natural Science Foundation of China (No. YQ2022A005).}





\begin{abstract}
We construct a slightly new noncommutative Calder\'on-Zygmund decomposition by further splitting the bad function.
Using this tool, we prove the weak type (1,1) boundedness of noncommutative Calder\'on-Zygmund operators under a class of $L_1$-integral conditions, which are close to $L_1$-Dini conditions.

\end{abstract}

\maketitle

\section{Introduction}

Calder\'on-Zygmund operators, a well-known class of singular integral operators, were introduced by Calder\'on and Zygmund \cite{Calderon1952On} in their seminal 1952 work.
The kernel of such an operator satisfies specific conditions, including the size and regularity conditions, the latter of which was refined by H\"ormander \cite{Hormander1960Estimates} in 1960.

The standard nonconvolution type Calder\'on-Zygmund operator is defined as
    \[
    Tf(x) = \int_{\R^d} K(x,y) f(y) \dif y, \quad x\notin \supp f,
    \]
where $K:\R^d \times \R^d \setminus \set{(x,x): x\in\R^d} \rightarrow \C$
is a standard nonconvolution type Calder\'on-Zygmund kernel, which is a binary function satisfying
    \begin{itemize}
	\item \textbf{Size condition:} there exists $C_1>0$ such that
	   \begin{align} \label{Size condition}
       |K(x,y)| \leq \frac{C_1}{|x-y|^d} ;
       \end{align}
	\item  \textbf{Lipschitz regularity condition:} there exists  $\gamma\in [0,1]$ and $C_2>0$ such that
	   \begin{align} \label{Lipschitz regularity condition}
	   |K(x,y)-K(x,z)|+|K(y,x)-K(z,x)| \leq \frac{C_2 |y-z|^{\gamma}}{|x-y|^{d+\gamma}}, \quad\text{if}\  |x-y|\geq 2|y-z|,
	   \end{align}
    \end{itemize}
where $x$, $y$, $z\in\R^d$.
The Lipschitz regularity condition (\ref{Lipschitz regularity condition}) can be relaxed to certain variants of smoothness conditions (see \cite{Grafakos2019Alimited, Suzuki2021TheCalderon, Watson1990Weighted}).
In the following, all Calder\'on-Zygmund kernels under consideration are of nonconvolution type.

Inspired by operator algebras, harmonic analysis, noncommutative geometry, and quantum probability (see \cite{Cadilhac2018Weak, Chen2013Harmonic, Junge2014Smooth, Junge2018Noncommutative, Mei2007Operator, Mei2009Pseudo-localization, Mei2017Free, Mei2022Mikhlin, Parcet2009Pseudo-localization}), noncommutative harmonic analysis has become an exciting field in recent years.
However, generalizing the Calder\'on-Zygmund decomposition, a key technique in harmonic analysis first introduced in \cite{Calderon1952On}, to the noncommutative setting is challenging.

In 2009, Parcet \cite{Parcet2009Pseudo-localization} rigorously constructed a kind of noncommutative Calder\'on-Zygmund decomposition by using Cuculescu's maximal weak type $(1,1)$ estimates for noncommutative martingales \cite{Cuculescu1971Martingales}.
With this tool, Parcet established the weak type $(1,1)$ estimates for operator-valued Calder\'on-Zygmund singular integrals with the kernel satisfying the Lipschitz regularity condition (\ref{Lipschitz regularity condition})
(Cadilhac later gave a shorter proof in \cite{Cadilhac2018Weak}).
Notably, the H\"ormander condition
    \begin{align}\label{Hormander condition}
	\sup_{y\in\R^d,|v|>0} \int_{|x-y|\geq 2|v|} |K(x,y+v)-K(x,y)| \dif x < \infty,
    \end{align}
which is weaker than (\ref{Lipschitz regularity condition}), already guarantees the weak type $(1,1)$ boundedness in classical Calder\'on-Zygmund theory.
It is therefore natural to ask whether the same still holds in the noncommutative setting?
This problem remains open.


A step forward was made in 2022 by Cadilhac et al. \cite{Cadilhac2022Spectral}.
They developed a more effective noncommutative Calder\'on-Zygmund decomposition without the off-diagonal term of the good function, which allowed them to prove the noncommutative weak type $(1,1)$ estimates for Calder\'on-Zygmund operators if the kernel $K$ satisfies
    \begin{align} \label{L2-integral regularity condition}
	\sum_{m\geq 1} \sup_{{\begin{subarray}{c} y\in Q \\ Q\ \text{dyadic} \end{subarray}}} \Big( 2^{md}\ell(Q)^d  \int_{2^m\ell(Q)\leq|x-c(Q)|\leq2^{m+1}\ell(Q)} |K(x,y)-K(x,c(Q))|^2 \dif x \Big)^{\frac{1}{2}}
	<\infty,
    \end{align}
where $\ell(Q)$ and $c(Q)$ stand for the length and the center of the dyadic cube $Q$, respectively.

More generally, for $1\leq q\leq\infty$, we say that a kernel $K$ satisfies the $L_q$-Dini condition if
    \begin{align} \label{sum delta < infty}
	\sum_{m\geq 1} \delta_q(m) <\infty,
    \end{align}
where
     \begin{align*}
	\delta_q(m):=
	 \sup_{{\begin{subarray}{c} y\in \R^d,\ R>0 \\ |v|<R \end{subarray}}} & \Big( (2^m R)^{d(q-1)} \int_{2^mR\leq |x-y|\leq 2^{m+1}R} |K(x,y+v)-K(x,y)|^q \dif x\Big)^{\frac{1}{q}},
    \end{align*}
for $m\geq 1$.
The fact that $\delta_1(m)\lesssim_d \delta_2(m)$ shows that the $L_2$-Dini condition ((\ref{sum delta < infty}) with $q=2$) is stronger than the $L_1$-Dini condition ((\ref{sum delta < infty}) with $q=1$), but weaker than the Lipschitz regularity condition (\ref{Lipschitz regularity condition}).
Note that the $L_1$-Dini condition ((\ref{sum delta < infty}) with $q=1$) corresponds to the H\"ormander condition (\ref{Hormander condition}).

Using the noncommutative Calder\'on-Zygmund decomposition in \cite{Cadilhac2022Spectral}, Hong, the first author, and Xu \cite{hong2023maximal} established the weak type (1,1) boundedness for the noncommutative maximal truncated Calder\'on-Zygmund operators under the $L_2$-Dini condition ((\ref{sum delta < infty}) with $q=2$).
Further results on weak (1,1) boundedness theory for the noncommutative Calder\'on-Zygmund operators can be found in \cite{Cadilhac2022Spectral, Lai2024Noncommutative, Mei2009Pseudo-localization}.

Since it remains open whether the $L_1$-Dini condition ((\ref{sum delta < infty}) with $q=1$) guarantees the weak type $(1,1)$ boundedness, a natural intermediate problem is whether we can lower the exponent from $q=2$ to some $q$ satisfying $1\leq q<2$?
This also appears to be quite challenging.
In this paper, we focus on a class of $L_1$-integral conditions, which is close to the $L_1$-Dini condition ((\ref{sum delta < infty}) with $q=1$).
For a dyadic cube $Q$, we define the kernel average by
    \[
    \mathcal{K}_Q(x,y):= \frac{1}{|Q|}\int_{|v|\leq \frac{\sqrt{d}}{2}\ell(Q)} |K(x,y+v)-K(x,y)| \dif v,
    \]
and set
    \[
    \mathcal{H}(m) := \sup_{Q\ \text{dyadic}} \int_{2^m\sqrt{d} \ell(Q)\leq |x-c(Q)| \leq 2^{m+1} \sqrt{d}\ell(Q)} \sup_{y\in Q} \mathcal{K}_Q(x,y) \dif x.
    \]
We then say that the kernel satisfies the $L_1$-integral condition if
    \begin{align}\label{L1 integral condition}
    \sum_{m\geq 1} \mathcal{H}(m) <\infty.
    \end{align}
Compared with the $L_1$-Dini condition ((\ref{sum delta < infty}) with $q=1$), the $L_1$-integral condition (\ref{L1 integral condition}) replaces the supremum over $v$ by an average over the ball $\{v: |v|\leq \frac{\sqrt{d}}{2}\ell(Q) \}$, while the supremum over $y$ is taken inside the integral of $x$.
The purpose of this paper is to show that this $L_1$-integral condition (\ref{L1 integral condition}) is already sufficient for the weak type (1,1) boundedness of the noncommutative Calder\'on-Zygmund operators.

Before stating the main result, we introduce some notions.
Let $\mM$ be a von Neumann algebra equipped with a normal semifinite faithful ($n. s. f. $, for short) trace $\tau$.
Let $\mN = L_{\infty}(\R^d)\bar\otimes \mM$ be the tensor von Neumann algebra with the tensor trace  $\phi = \int\otimes \tau$.
Denote the noncommutative $L_p$ space associated with $(\mN, \phi)$ by $L_p(\mN)$ for $1\leq p\leq\infty$ (see Appendix \ref{Appendix: Noncommutative Lp spaces and maximal norms} for details).
The noncommutative Calder\'on-Zygmund operator $T$ with a kernel $K$ is defined as
    \begin{align}\label{def of Tf}
	T f(x) = \int_{\R^d} K(x,y) f(y) \dif y, \quad x\notin \overrightarrow{\supp}\ f,
    \end{align}
where $f\in L_1(\mN)\bigcap L_{\infty}(\mN)$ is compactly supported and measurable.
Here, $\overrightarrow{\text{supp}}\ f$ denotes the support of $f$ as an operator-valued function in $\R^d$ rather than the support projection as an element in a von Neumann algebra.
For any $\epsilon>0$, the associated truncated singular integral $T_{\epsilon}f$ is defined as
    \begin{align}\label{def of T epsilon f}
	T_{\epsilon}f(x) = \int_{|x-y|>\epsilon} K(x,y) f(y) \dif y.
    \end{align}

The main result is stated as follows.
    \begin{theorem}\label{Main theorem}
	Suppose $T$ is a Calder\'on-Zygmund operator defined as in (\ref{def of Tf}) associated with a kernel satisfying (\ref{Size condition}) and the $L_1$-integral condition (\ref{L1 integral condition}).
	Let $T_{\epsilon}$ be defined as in (\ref{def of T epsilon f}).
	Suppose $(T_{\epsilon})_{\epsilon>0}$ is of strong type $(p_0,p_0)$ for some $1< p_0 <\infty$, that is, for any $f\in L_{p_0}(\mN)$,
	\[
	\norm{(T_{\epsilon}f)_{\epsilon>0}}_{L_{p_0}(\mN;\ell_{\infty})} \lesssim \norm{f}_{L_{p_0}(\mN)}.
	\]
	Then, $(T_{\epsilon})_{\epsilon>0}$ is of weak type $(1,1)$, that is, for any $f\in L_1(\mN)$,
	\[
	\norm{(T_{\epsilon}f)_{\epsilon>0}}_{\Lambda_{1,\infty}(\mN;\ell_{\infty})} \lesssim \norm{f}_{L_1(\mN)}.
	\]
We refer to see Appendix \ref{Appendix: Noncommutative Lp spaces and maximal norms} for the definitions of the noncommutative maximal norm $\norm{\cdot}_{L_p(\mN; \ell_{\infty})}$ and the noncommutative weak maximal norm $\norm{\cdot}_{\Lambda_{p,\infty}(\mN;\ell_{\infty})}$.
    \end{theorem}

Our main strategy in this paper is to develop a refined noncommutative Calder\'on-Zygmund decomposition, along with related techniques, adapted to kernels satisfying the $L_1$-integral condition (\ref{L1 integral condition}).
Existing methods are inadequate for this purpose.
Parcet's decomposition \cite{Parcet2009Pseudo-localization} requires the stronger Lipschitz regularity condition (\ref{Lipschitz regularity condition}) and may not be directly applied to the noncommutative maximal operators.
Meanwhile, the decomposition by Cadilhac et al. \cite{Cadilhac2022Spectral} can only deal with the
$L_2$-Dini condition ((\ref{sum delta < infty}) with $q=2$) and therefore does not apply to our setting.
To address this gap, we introduce a slightly new noncommutative Calder\'on-Zygmund decomposition achieved by further splitting of the bad function into a convolution part and a remainder part.
In addition, we avoid using the H\"older inequality for the noncommutative square functions but instead employ more explicit factorizations for such terms.

This paper is organized as follows.
We begin by presenting our key tool: the refined noncommutative Calder\'on-Zygmund decomposition (Theorem \ref{nc CZ decomposition with convolution}) in Section \ref{Section: Noncommutative C-Z decomposition}.
Next, Section \ref{Section: Two reductions} starts the proof of the main result (Theorem \ref{Main theorem}), reducing it to proving the weak type (1,1) estimates for the lacunary sequences with real kernels (Theorem \ref{Second theorem}).
With this reduction established, Section \ref{Proof of the main theorem} is devoted to proving Theorem \ref{Second theorem}.
Finally, some related background concepts are included in Appendix \ref{Appendix: Noncommutative Lp spaces and maximal norms}.

Now, we conclude this section by listing the notation required for our later analysis.\\

\textbf{Notation}
\begin{itemize}
	\item $\mQ_n$: the set of all dyadic cubes of side length $2^{-n}$, $n\in\Z$.
	\item $Q_{x,n}$: the unique cube in $\mQ_n$ containing $x\in\R^d$.
	\item $\mQ$: the set of all standard dyadic cubes in $\R^d$, that is, $\mQ = \bigcup\limits_n\mQ_n$.
	\item $c(Q)$: the center of the cube $Q$.
	\item $lQ$: the concentric cube sharing the center of  $Q$ such that its length is $l$ times the length of $Q$.
	\item $B_r(x)$: the ball with $x\in\R^d$ as its center and $r>0$ as its radius, and the center $x$ will be omitted when it is the origin.
	\item $|E|$: the volume of set $E\subset\R^d$.
	\item $[k]$: the integer part of $k$.

\end{itemize}

\section{Noncommutative Calder\'on-Zygmund decomposition}
\label{Section: Noncommutative C-Z decomposition}

The Calder\'on-Zygmund decomposition is a fundamental tool for studying the boundedness of singular integral operators.
When transferring to the von Neumann algebraic setting, its noncommutative counterpart is primarily based on Cuculescu's martingale projections \cite{Cuculescu1971Martingales}.
 In this section, we first review Cuculescu's construction, and then use it to develop our refined version of the noncommutative Calder\'on-Zygmund decomposition.
Background on noncommutative $L_p$ spaces and maximal norms can be found in Appendix \ref{Appendix: Noncommutative Lp spaces and maximal norms}.

\subsection{Cuculescu's construction}

Let $\mN = L_{\infty}(\R^d) \bar{\otimes} \mM$ be the von Neumann algebra tensor product, equipped with the tensor trace $\phi = \int_{\R^d} \otimes \tau$.
For $n \in \mathbb{Z}$, let $\sigma_n$ be the $\sigma$-algebra generated by $\mQ_n$, and let $\mN_n = L_{\infty}(\R^d, \sigma_n, dx) \bar{\otimes} \mM$, which is a von Neumann subalgebra of $\mN$.
Then, $(\mN_n)_{n \in \mathbb{Z}}$ forms a sequence of increasing von Neumann subalgebras of $\mN$ whose union is weak-$*$ dense in $\mN$.

Let $E_n$ be a conditional expectation from $\mN$ onto $\mN_n$.
A sequence $(a_n)_{n\in\Z} \subset \mN$ is said to be a martingale if $E_{n}(a_{n+1}) = a_n$.
For $1\leq q \leq\infty$, if in addition $(a_n)_{n\in\Z} \subset L_q(\mN)$, then $(a_n)_{n\in\Z}$ is called an $L_q$-martingale.
In this paper, we define the conditional expectation $E_n$ by
    \begin{align*}
	E_n (f) = \sum_{Q\in\mQ_n} f_Q \chi_Q,\quad \forall f\in L_1(\mN),
    \end{align*}
where $\chi_Q$ is the characteristic function of $Q$, and $f_Q = \frac{1}{|Q|} \int_{Q} f(y) \dif y $.
Thus, $(\mN_n)_{n\in\Z}$ becomes a filtration with the associated conditional expectations  $(E_n)_{n\in\Z}$.
Define $f_n= E_n(f)$.
Then $(f_n)_{n\in\Z}$ is clearly an $L_1$-martingale.

Set
    \begin{align*}
	\mN_{c,+} = \set{f: \R^d \rightarrow \mM \cap L_1(\mM): f\geq 0,\ \overrightarrow{\text{supp}}\ f \text{ is compact}},
    \end{align*}
where $\overrightarrow{\text{supp}}\ f$ means the support of $f$ as an operator-valued function in $\R^d$.
Based on the density of $\mN_{c,+}$ in $L_1(\mN)_+$, we only need to consider the functions in $\mN_{c,+}$.

The following lemma, established in \cite{Cuculescu1971Martingales} (see also \cite{Parcet2009Pseudo-localization}), plays an important role in the noncommutative Calder\'on-Zygmund decomposition.
    \begin{lemma}$(\rm{Cuculescu})$ \label{cuculescu's construction}
	Let $f\in\mN_{c,+}$.
	For any $\lambda>0$, there exists a decreasing sequence $(q_n(f,\lambda))_{n\in\Z}$  (for convenience, we write $q_n(f,\lambda)$ as $q_{n}$) of projections in $\mN$ satisfying
	    \begin{enumerate}[(i)]
	    \item $q_n \in \mN_n$;
		\item $q_{n}$ commutes with $q_{n-1} f_n q_{n-1}$;
		\item $q_n f_n q_n \leq \lambda q_n$;
		\item The following estimate holds
	              \begin{align*}
	              \phi\bigg(\Bone_{\mN} - \bigwedge_{n\in\Z} q_{n} \bigg)
	              \leq \frac{1}{\lambda} \norm{f}_{L_1(\mN)}.
	              \end{align*}
	    \end{enumerate}
    \end{lemma}
We see that $f_n$ can be bounded by any $\lambda>0$ if $n$ is small enough.
Hence, for a given $\lambda>0$, there exists an integer $n_{\lambda}(f)$ such that $f_n \leq \lambda \Bone_{\mN}$ when $n\leq n_{\lambda}(f)$.
The decreasing sequence of projections $(q_n)_{n\in\Z}$ in Lemma \ref{cuculescu's construction} is in fact given by
    \begin{align*}
	q_n=
    \begin{cases}
    \Bone_{\mN}, & \text{if}\ n\leq n_{\lambda}(f),\\
    \chi_{(0,\lambda](q_{n-1} f_n q_{n-1})}, & \text{if}\ n > n_{\lambda}(f).
    \end{cases}
    \end{align*}
The following alternative representation of $q_n$ is more useful:
    \begin{align*}
	q_n = \sum_{Q\in\mQ_n} q_Q \chi_Q,
    \end{align*}
where
    \begin{align*}
	q_Q=
    \begin{cases}
    \Bone_{\mN}, & \text{if}\ n\leq n_{\lambda}(f),\\
    \chi_{(0,\lambda](q_{\widehat Q} f_Q q_{\widehat Q})}, & \text{if}\ n > n_{\lambda}(f).
    \end{cases}
    \end{align*}
Here, $\widehat Q$ denotes the dyadic father of $Q$.
One can easily verify that
    \begin{align}\label{q_Q f_Q q_Q < lambda q_Q}
	q_Q \leq q_{\widehat Q}, \quad q_Q \text{ commutes with } q_{\widehat Q} f_Q q_{\widehat Q}, \quad q_Q f_Q q_Q \leq \lambda q_Q.
    \end{align}
Furthermore, let $(p_n)_{n\in\Z}$ be a sequence of projections defined by
    \begin{align*}
	p_n = q_{n-1} - q_n = \sum_{Q\in\mQ_n} (q_{\widehat Q}-q_Q) \chi_Q =: \sum_{Q\in\mQ_n}p_Q\chi_Q.
    \end{align*}
This gives
    \begin{align} \label{sum pn=1-q}
	\sum_{n\in\Z} p_n = \Bone_{\mN} - q = q^{\perp} \quad\text{with}\quad q= \bigwedge_{n\in\Z}q_n.
    \end{align}

The next result follows directly from Lemma \ref{cuculescu's construction}; we omit the proofs here.

    \begin{lemma} \label{properties of pn, qn}
    Let $(q_n)_{n\in\Z}$ and $(p_n)_{n\in\Z}$ be the sequences of projections in Lemma \ref{cuculescu's construction}.
    Then, the following equalities hold:
	\begin{enumerate}[(i)]
		\item $p_n f_n q = q f_n p_n = 0$, for $n\in\Z$.
		\item $p_n f_{n\wedge k} p_k = p_k f_{n\wedge k} p_n = 0$, for $n$, $k\in\Z$, $n\neq k$, where $n\wedge k := \min\set{n,k}$.
	\end{enumerate}
    \end{lemma}

\subsection{Noncommutative Calder\'on-Zygmund decomposition}

    For fixed $f\in\mN_{c,+}$, $\lambda>0$, and $s\in\N$, we define the projection $\zeta_s$ as
    \begin{align}\label{projection zeta_s}
	\zeta_s = \bigg(\bigvee_{Q\in\mQ} p_Q \chi_{(2s+1)Q}\bigg)^{\perp}.
    \end{align}
    We require the following property of this projection, whose proof can be found in \cite{Cadilhac2022Spectral}.

    \begin{lemma} \label{boundedness of 1-zeta}
    For fixed $f\in\mN_{c,+}$, $\lambda>0$, and $s\in\N$, we have
	\begin{align*} 
	\phi(\Bone_{\mN}-\zeta_s) \leq \frac{(2s+1)^d}{\lambda}  \norm{f}_{L_1(\mN)}.
	\end{align*}
    \end{lemma}



Below, we construct a slightly new noncommutative Calder\'on-Zygmund decomposition inspired by \cite{Cadilhac2022Spectral}, and present its properties later.
    \begin{theorem} \label{nc CZ decomposition with convolution}
	Fix $f\in\mN_{c,+}$, $\lambda>0$, and $s\in\N$.
	Let $(q_n)_{n\in\Z}$ and $(p_n)_{n\in\Z}$ be the sequences of projections in Lemma \ref{cuculescu's construction}, where $q_n = \sum\limits_{Q\in\mQ_n} q_Q \chi_Q$, $p_n = \sum\limits_{Q\in\mQ_n} p_Q \chi_Q$.
Suppose $h\in C_{c}^{\infty} (\R^d)$ satisfies
    \begin{align*}
    \supp h \subset B_1,\quad \int_{\R^d} h =1, \quad h\geq 0.
    \end{align*}
Define $h_r(x)= \frac{1}{r^d} h(\frac{x}{r})$.
Then, $f$ admits the decomposition:
	\begin{align*}
	f = g+ b
	  = g + \widetilde{b}_{\text{d}} + b_{\text{d}}^0 + \widetilde{b}_{\text{off}} + b_{\text{off}}^0,
	\end{align*}
	where $g$ is positive and $b$ is self-adjoint in $\mN$.
	More precisely, these components are defined as follows:
	    \[g=qfq + \sum\limits_{n\in\Z}p_n f_n p_n,\]
        \begin{gather*}
        \widetilde{b}_{\text{d}} = \sum\limits_{n\in\Z} \widetilde{b}_{\text{d},n} = \sum\limits_{n\in\Z}\sum\limits_{Q\in\mQ_n}\widetilde{b}_{\text{d},n}^Q ,\quad
        b_{\text{d}}^0 = \sum\limits_{n\in\Z} b_{\text{d},n}^0=\sum\limits_{n\in\Z} \sum\limits_{Q\in\mQ_n}(b_{\text{d},n}^Q-\widetilde{b}_{\text{d},n}^Q),\\
        \widetilde{b}_{\text{off}} = \sum\limits_{n\in\Z} \widetilde{b}_{\text{off},n} = \sum\limits_{n\in\Z} \sum\limits_{Q\in\mQ_n} \widetilde{b}_{\text{off},n}^Q ,\quad
        b_{\text{off}}^0 = \sum\limits_{n\in\Z} b_{\text{off},n}^0 = \sum\limits_{n\in\Z}\sum\limits_{Q\in\mQ_n}(b_{\text{off},n}^Q-\widetilde{b}_{\text{off},n}^Q),
        \end{gather*}
where for each $Q\in\mQ_n$,
        \begin{gather*}
     	b_{\text{d},n}^Q = p_Q(f-f_n)p_Q\chi_Q,
     	\quad b_{\text{off},n}^Q = p_Q(f-f_n)q_Q\chi_Q + q_Q(f-f_n)p_Q\chi_Q, \\
     	\widetilde{b}_{\text{d},n}^Q = b_{\text{d},n}^Q * h_{\frac{\sqrt{d}}{2}\ell(Q)},
     	\quad \widetilde{b}_{\text{off},n}^Q = b_{\text{off},n}^Q * h_{\frac{\sqrt{d}}{2}\ell(Q)}.
        \end{gather*}
    \end{theorem}

    \begin{proof}
Set $p_{\infty}=q$, $f_{\infty}=f$, and $\widehat{\Z} = \Z\cup\set{\infty}$.
By (\ref{sum pn=1-q}), $f$ can be rewritten as
    \begin{align*}
    	f = \sum_{n\in\widehat{\Z}}\sum_{k\in\widehat{\Z}} p_n f_{n\wedge k}p_k + \sum_{n\in\widehat{\Z}}\sum_{k\in\widehat{\Z}} p_n(f-f_{n\wedge k})p_k
    	 =: g+b.
    \end{align*}

    For the good function $g$, Lemma \ref{properties of pn, qn} implies
    \begin{align*}
    g = qfq + \sum_{k\in\Z}qf_k p_k + \sum_{n\in\Z}p_n f_n q + \sum_{n\in\Z}\sum_{k\in\Z} p_n f_{n\wedge k}p_k
      = qfq + \sum_{n\in\Z}p_n f_n p_n.
    \end{align*}

For the bad function $b$,  we consider those terms with finite indices $n$ and $k$ satisfying $n\neq k$.
Since $p_k=q_{k-1}-q_k$, a direct calculation shows
    \begin{gather*}
    \sum_{n\in\Z}\sum_{k>n}p_n(f-f_{n})p_k = \sum_{n\in\Z}p_n(f-f_{n})(q_n-q), \\
    \sum_{n\in\Z}\sum_{k<n}p_n(f-f_{k})p_k
    = \sum_{k\in\Z}\sum_{n>k}p_n(f-f_{k})p_k = \sum_{k\in\Z}(q_k-q)(f-f_k)p_k,
    \end{gather*}
which implies that
    \begin{align*}
     b = \sum_{n\in\Z} p_n(f-f_n)p_n + \sum_{n\in\Z} \big(p_n(f-f_n)q_n + q_n(f-f_n)p_n \big)
    \end{align*}
by using Lemma \ref{properties of pn, qn} again.
Next, we insert a convolution operator to decompose $b$ further.
Let $h_r$ be defined as above.
For a fixed cube $Q\in\mQ_n$, let
    \begin{align*}
	\left\{\begin{array}{ll}
    b_{\text{d},n}^Q = p_Q(f-f_n)p_Q\chi_Q, \\
    \widetilde{b}_{\text{d},n}^Q = b_{\text{d},n}^Q * h_{\frac{\sqrt{d}}{2}\ell(Q)},
    \end{array}\right.
    \quad \text{and}\quad
    \left\{\begin{array}{ll}
    b_{\text{off},n}^Q = p_Q(f-f_n)q_Q\chi_Q + q_Q(f-f_n)p_Q\chi_Q, \\
    \widetilde{b}_{\text{off},n}^Q = b_{\text{off},n}^Q * h_{\frac{\sqrt{d}}{2}\ell(Q)}.
    \end{array}\right.
    \end{align*}
By these definitions as well as the disjointness of dyadic cubes, we can express the bad function $b$ as
    \begin{align*}
    \underbrace{\sum_{n\in\Z}\sum_{Q\in\mQ_n} \widetilde{b}_{\text{d},n}^Q}_{\widetilde{b}_{\text{d}} = \sum\limits_{n\in\Z} \widetilde{b}_{\text{d},n}}
    + \underbrace{\sum_{n\in\Z}\sum_{Q\in\mQ_n}(b_{\text{d},n}^Q-\widetilde{b}_{\text{d},n}^Q)}_{b_\text{d}^0 = \sum\limits_{n\in\Z}b_{\text{d},n}^0}
    + \underbrace{\sum_{n\in\Z}\sum_{Q\in\mQ_n}\widetilde{b}_{\text{off},n}^Q}_{\widetilde{b}_{\text{off}} =  \sum\limits_{n\in\Z} \widetilde{b}_{\text{off},n}}
    + \underbrace{\sum_{n\in\Z}\sum_{Q\in\mQ_n}(b_{\text{off},n}^Q-\widetilde{b}_{\text{off},n}^Q)}_{b_{\text{off}}^0 = \sum\limits_{n\in\Z} b_{\text{off},n}^0}.
    \end{align*}

    This yields the desired decomposition $f = g+ b = g + \widetilde{b}_{\text{d}} + b_{\text{d}}^0 + \widetilde{b}_{\text{off}} + b_{\text{off}}^0$; the positivity of $g$ and self-adjointness of $b$ follows immediatly from  their constructions.
    \end{proof}


The decomposition in Theorem \ref{nc CZ decomposition with convolution} admits the following properties:

    \begin{lemma}\label{properties of nc CZ decomposition with convolution}
    Let $\zeta_s$ be the projection given as in (\ref{projection zeta_s}).
    Then,
		\begin{enumerate}[(i)]
		\item $\norm{g}_{L_1(\mN)}\leq\norm{f}_{L_1(\mN)}$ and $\norm{g}_{L_{\infty}(\mN)}\leq 2^d\lambda$.
		\item For each $Q\in\mQ_n$, $\int_{\R^d} \widetilde{b}_{\text{d},n}^Q = 0$, 
		   and for any $x,$ $y\in\R^d$, $\zeta_s(x) \widetilde{b}_{\text{d},n}(y) \zeta_s(x) = 0$ when $|x-y|\leq (s-\sqrt{d})2^{-n}$.
		   Furthermore,
		   $\sum\limits_{n\in\Z}\sum \limits_{Q\in\mQ_n} \|b_{\text{d},n}^Q\|_{L_1(\mN)} \lesssim \norm{f}_{L_1(\mN)}$.
		\item For any $x,$ $y\in\R^d$, $\zeta_s(x) b_{\text{d},n}^0(y) \zeta_s(x) = 0$ when $|x-y|\leq (s-\sqrt{d})2^{-n}$.
		\item For each $Q\in\mQ_n$, $\int_{\R^d} \widetilde{b}_{\text{off},n}^Q = 0$, 
		   and for any $x,$ $y\in\R^d$, $\zeta_s(x) \widetilde{b}_{\text{off},n}(y) \zeta_s(x) = 0$ when $|x-y|\leq (s-\sqrt{d})2^{-n}$. 
		\item For any $x,$ $y\in\R^d$, $\zeta_s(x) b_{\text{off},n}^0(y) \zeta_s(x) = 0$ when $|x-y|\leq (s-\sqrt{d})2^{-n}$.
	    \end{enumerate}
    \end{lemma}

\begin{proof}
$(i)$ follows directly from Parcet's noncommutative Calder\'on-Zygmund decomposition, as detailed in \cite{Cadilhac2018Weak} or \cite{Cadilhac2022Spectral}.

Next, consider $(ii)$.
The proof contains three parts: the cancellation condition, the support condition, and the boundedness.

By the definition of $b_{\text{d},n}^Q$ in Theorem \ref{nc CZ decomposition with convolution}, we have $\int_{\R^d} b_{\text{d},n}^Q = 0$ for each $Q\in\mQ_n$.
Since $\int_{\R^d} h_{\frac{\sqrt{d}}{2}\ell(Q)} = 1$, it follows that
    \begin{align*}
	\int_{\R^d} \widetilde{b}_{\text{d},n}^Q(x) \dif x
	= \int_{\R^d} b_{\text{d},n}^Q(y) \Big(\int_{\R^d} h_{\frac{\sqrt{d}}{2}\ell(Q)}(x-y) \dif x \Big) \dif y
	=0,
    \end{align*}
which is the required cancellation property for $\widetilde{b}_{\text{d},n}^Q$.

Now we consider the support condition.
For each $Q\in\mQ$, the construction of $\zeta_s$ in (\ref{projection zeta_s}) implies
    \begin{align} \label{zeta(x)pQ=0}
	\zeta_s(x)p_Q =p_Q\zeta_s(x) = 0 \quad\text{if}\quad x\in (2s+1)Q.
    \end{align}
Indeed, if $x\in (2s+1)Q$, then
    \begin{align*}
	\zeta_s(x)
	= 1- \bigvee_{Q\in\mQ} p_Q \chi_{(2s+1)Q}(x)
	\leq 1-p_Q,
    \end{align*}
implying
    \begin{align*}
	\zeta_s(x)p_Q = \zeta_s(x) (1-p_Q)p_Q=0,
    \end{align*}
and the same holds for $p_Q\zeta_s(x)$, which lead to (\ref{zeta(x)pQ=0}).
Then, note that $\supp(h_{\frac{\sqrt{d}}{2}\ell(Q)})\subset B_{\frac{\sqrt{d}}{2}\ell(Q)}$, we write $\zeta_s(x) \widetilde{b}_{\text{d},n}(y) \zeta_s(x)$ as follows:
    \begin{align*}
    &\zeta_s(x) \bigg( \sum_{Q\in\mQ_n}\int_{ B_{\frac{\sqrt{d}}{2}\ell(Q)} } b_{\text{d},n}^Q(u) h_{\frac{\sqrt{d}}{2}\ell(Q)}(y-u) \dif u \bigg) \zeta_s(x)\\
    &\quad = \sum_{Q\in\mQ_n}\int_{u\in \big(B_{\frac{\sqrt{d}}{2}\ell(Q)}(y)\bigcap Q\big)}  \Big(\zeta_s(x) p_{Q}(f-f_n)(u) p_{Q} \zeta_s(x) \Big)  h_{\frac{\sqrt{d}}{2}\ell(Q)}(y-u) \dif u.
    \end{align*}
As mentioned, we let $Q_{u,n}$ be the unique cube in $\mQ_n$ that contains $u$.
Hence, with (\ref{zeta(x)pQ=0}), it is enough to show that
    \begin{align} \label{x in (2s+1)Qun}
	x\in (2s+1)Q_{u,n} \quad\text{if}\quad |x-y|\leq (s-\sqrt{d})2^{-n}.
    \end{align}
In fact, for $u\in B_{\frac{\sqrt{d}}{2}\ell(Q_{u,n})}(y)\bigcap Q_{u,n}$,
    \begin{align*}
	|x-c(Q_{u,n})|
	&\leq |x-y| + |y-u| + |u-c(Q_{u,n})| \\
	&\leq (s-\sqrt{d})\ell(Q_{u,n}) + \frac{\sqrt{d}}{2}\ell(Q_{u,n}) + \frac{\sqrt{d}}{2}\ell(Q_{u,n})\\
	&= s \ell(Q_{u,n})
    \end{align*}
implying (\ref{x in (2s+1)Qun}).

The boundedness of $b_{\text{d},n}^Q$ results from the fact that $\phi(f p_n) = \phi(f_n p_n)$ together with (\ref{sum pn=1-q}).
This means,
    \begin{align*}
    \sum_{n\in\Z}\sum_{Q\in\mQ_n}\|b_{\text{d},n}^Q\|_{L_1(\mN)}
    &\leq \sum_{n\in\Z}\sum_{Q\in\mQ_n} \big(\phi(fp_Q\chi_Q) + \phi(f_n p_Q\chi_Q)\big) \\
    &= \sum_{n\in\Z} \big(\phi(fp_n) + \phi(f_n p_n)\big) \\
    &\leq 2\phi(f).
    \end{align*}
This completes the proof of $(ii)$.

$(iii)$
By the support condition in $(ii)$, we may restrict our attention to the nonconvolution part of $b_{\text{d},n}^0$.
That is, by writing
    \begin{align*}
    \zeta_s(x) b_{\text{d},n}^0 \zeta_s(x) = \sum\limits_{Q\in\mQ_n}\zeta_s(x) b_{\text{d},n}^Q(y) \zeta_s(x) - \sum\limits_{Q\in\mQ_n}\zeta_s(x) \widetilde{b}_{\text{d},n}^Q(y) \zeta_s(x),	
    \end{align*}
we only have to estimate the first sum, since the second one has already been considered in $(ii)$.
The required estimate is obtained by using the same argument as in $(ii)$, and we therefore omit the details.

Finally, the proofs of $(iv)$ and $(v)$ follow from analogous reasoning and are also omitted.

\end{proof}

\begin{remark}\label{the value of s}
We will fix $s= [50\sqrt{d}]$ throughout this paper.
Furthermore, the projection $\zeta_s$ is denoted simply by $\zeta$ when $s$ is fixed.
\end{remark}


\begin{remark}\label{sums of bd from 1 to infty}
Recall the construction of $q_n$: there exists an integer $n_{\lambda}(f)$ such that $q_n = 1_{\mN}$ when $n\leq n_{\lambda}(f)$.
Without loss of generality, we may set $n_{\lambda}(f) = 0$.
This leads to
     \begin{align*}
	 \widetilde{b}_{\text{d}} = \sum_{n\geq 1} \widetilde{b}_{\text{d},n}, \quad
	 \ b_{\text{d}}^0 = \sum_{n\geq 1} b_{\text{d},n}^0, \quad
	 \ \widetilde{b}_{\text{off}} = \sum_{n\geq 1} \widetilde{b}_{\text{off},n}, \quad
	 \ b_{\text{off}}^0 = \sum_{n\geq 1} b_{\text{off},n}^0,
     \end{align*}
thereby simplifying the summation.
\end{remark}


\section{Two reductions}\label{Section: Two reductions}

This section introduces two key lemmas (Lemma \ref{reduction to real kernel} and Lemma \ref{reduction to lacunary sequence}), which reduce the proof of the main theorem to the cases of real kernels and lacunary sequences.
For their proofs, we refer the reader to \cite{hong2023maximal}.

\textbf{Step 1.}
Let $T_{\epsilon}$ be a noncommutative Calder\'on-Zygmund truncated operator with a complex kernel $K$ (defined as in (\ref{def of T epsilon f})).
We decompose $K$ into its real and imaginary parts: $\Re(K)$ and $\Im(K)$.
The truncated operators associated with $\Re(K)$ and $\Im(K)$ are then given by
    \begin{align*}
	\Re(T_{\epsilon})f(x) = \int_{|x-y|>\epsilon} \Re(K)(x,y) f(y) \dif y
    \end{align*}
and
    \begin{align*}
	\Im(T_{\epsilon})f(x) = \int_{|x-y|>\epsilon} \Im(K)(x,y) f(y) \dif y.
    \end{align*}
\begin{lemma}\label{reduction to real kernel}
	Let $T$ be a noncommutative Calder\'on-Zygmund operator (defined as in (\ref{def of Tf})) whose kernel $K$ satisfies (\ref{Size condition}) and the $L_1$-integral condition (\ref{L1 integral condition}).
	Let $T_{\epsilon}$ be defined as in (\ref{def of T epsilon f}).
	Then,
	\begin{enumerate}[(i)]
	\item Both $\Re(K)$ and $\Im(K)$ satisfy (\ref{Size condition}) and (\ref{L1 integral condition}).
	\item If $(T_{\epsilon})_{\epsilon>0}$ is of strong type $(p_0,p_0)$ for some $p_0\in (1,\infty)$, then, both $(\Re(T_{\epsilon}))_{\epsilon>0}$ and $(\Im(T_{\epsilon}))_{\epsilon>0}$ are of strong type $(p_0,p_0)$.
	\item If $(\Re(T_{\epsilon}))_{\epsilon>0}$ and $(\Im(T_{\epsilon}))_{\epsilon>0}$ are of weak type $(1,1)$, so is $(T_{\epsilon})_{\epsilon>0}$.
	\end{enumerate}
\end{lemma}
By this lemma, we may assume that the kernel $K$ is real throughout this paper.

\textbf{Step 2.}
Let $\Phi$ be a smooth radial nonnegative function on $\R^d$ such that \begin{align*}
	\supp\Phi\subset \big\{x\in\R^d: 1/2 \leq |x| \leq 2 \big\}
	\quad\text{and}\quad
	\sum_{j\in\Z}\Phi_j(x) = 1,\ x\in\R^d\setminus \{0\},
    \end{align*}
where $\Phi_j(x):= \Phi(\frac{x}{2^j\sqrt{d}})$.
Thus, the truncated singular integrals $T_{\epsilon}f$ can be written as
    \begin{align*}
	T_{\epsilon}f
	= \sum_{j\in\Z} \int_{|x-y|>\epsilon} K(x,y) \Phi_j(x-y)f(y) \dif y.
    \end{align*}
From the conditions $|x-y| > \epsilon$ and $\supp (\Phi_j) \subset \{ x \in \mathbb{R}^d : 2^{j-1}\sqrt{d} \leq |x| \leq 2^{j+1}\sqrt{d} \}$, we see that for $i_{\epsilon} = \big[ \log_2 (\frac{\epsilon}{2\sqrt{d}}) \big] +1$, the following equality holds (when $\epsilon > 0$ is sufficiently small):
    \begin{align*}
	T_{\epsilon}f(x)
	= \sum_{j\geq  i_{\epsilon}} \int_{\R^d} K(x,y) \Phi_j(x-y)f(y) \dif y
	+ \int_{|x-y|>\epsilon} K(x,y) \Phi_{i_{\epsilon}}(x-y)f(y) \dif y.
    \end{align*}
It is convenient to introduce the following notation:
    \begin{gather}
     \label{Kj}K_j(x,y) = K(x,y)\Phi_j(x-y), \\
     \label{Sif(x)} S_i f(x)  =  \sum_{j\geq i} \int_{\R^d} K_j(x,y)f(y) \dif y,  \\
     \label{T epsilon,i epsilon f(x)} T_{\epsilon}^{i_{\epsilon}} f(x) = \int_{|x-y|>\epsilon} K_{i_{\epsilon}}(x,y)f(y) \dif y.
    \end{gather}
Hence,
    \begin{align*}
	T_{\epsilon}f(x) = S_{i_{\epsilon}}f(x) + T_{\epsilon}^{i_{\epsilon}} f(x).
    \end{align*}

    \begin{lemma}\label{reduction to lacunary sequence}
	Let $T$ be a noncommutative Calder\'on-Zygmund operator (defined as in (\ref{def of Tf})) whose kernel $K$ satisfies (\ref{Size condition}) and the $L_1$-integral condition (\ref{L1 integral condition}).
	Let $T_{\epsilon}^{i_{\epsilon}}$ be defined as in (\ref{T epsilon,i epsilon f(x)}).
	Then
	\begin{enumerate}[(i)]
	\item $(T_{\epsilon}^{i_{\epsilon}})_{\epsilon>0}$ is of strong type $(p,p)$ for $1<p\leq\infty$.
	\item $(T_{\epsilon}^{i_{\epsilon}})_{\epsilon>0}$ is of weak type $(1,1)$.
	\end{enumerate}
    \end{lemma}

The preceding two lemmas imply that Theorem \ref{Main theorem} follows directly from the weak type (1,1) boundedness of $S_i$ given in (\ref{Sif(x)}).
Thus, it is enough to prove
\begin{theorem}\label{Second theorem}
	Let $T$ be a noncommutative Calder\'on-Zygmund operator defined as in (\ref{def of Tf}) with the real kernel $K$ satisfying (\ref{Size condition}) and the $L_1$-integral condition (\ref{L1 integral condition}).
	Let $S_i$ be defined as in (\ref{Sif(x)}).
	Then, the sequence of operators $(S_i)_{i\in\Z}$ is of weak type $(1,1)$, that is, for any $f\in L_1(\mN)$,
	\[
	\norm{(S_i f)_{i\in\Z}}_{\Lambda_{1,\infty}(\mN;\ell_{\infty})} \lesssim \norm{f}_{L_1(\mN)}.
	\]
	More precisely, for any $f\in L_1(\mN)$ and $\lambda>0$, there exists a projection $e\in\mN$ such that
	\[
	\sup_{i\in\Z}\norm{e S_if e}_{L_{\infty(\mN)}} \leq \lambda \quad\text{and}\quad  \phi(e^{\bot})\lesssim \lambda^{-1} \norm{f}_{L_1(\mN)} .
	\]
\end{theorem}

\section{The proof of Theorem \ref{Second theorem}}\label{Proof of the main theorem}
Note that every operator in a von Neumann algebra can be decomposed into a linear combination of four positive elements.
By this fact and the density of $\mN_{c,+}$ in $L_1(\mN)_+$, we may assume $f\in\mN_{c,+}$ without loss of generality.
Now, consider a fixed $f \in \mN_{c,+}$ and $\lambda > 0$.
We apply Theorem \ref{nc CZ decomposition with convolution} to decompose $f = g + b$. Then, by the quasi-triangle inequality, it suffices to find projections $e_1, e_2 \in \mN$ such that
    \begin{align}
    \label{estimate of good function}
	\sup_{i\in\Z}\norm{e_1 S_ig e_1}_{L_{\infty}(\mN)} \leq \lambda \quad\text{and}\quad \phi(e_1^{\bot})\lesssim \lambda^{-1}\norm{f}_{L_1(\mN)}, \\
	\label{estimate of bad function}
	\sup_{i\in\Z}\norm{e_2 S_ib e_2}_{L_{\infty}(\mN)} \leq \lambda \quad\text{and}\quad \phi(e_2^{\bot})\lesssim \lambda^{-1}\norm{f}_{L_1(\mN)}.
    \end{align}
Indeed, by setting $e = e_1\wedge e_2$, we obtain
    \begin{align*}
	\norm{e S_if e}_{L_{\infty}(\mN)}
	\leq \norm{e S_ig e}_{L_{\infty}(\mN)} + \norm{e S_ib e}_{L_{\infty}(\mN)}
	\leq 2\lambda
    \end{align*}
and
    \begin{align*}
	\phi(e^{\bot}) \leq \phi(e_1^{\bot}) + \phi(e_2^{\bot}) \lesssim \lambda^{-1}\norm{f}_{L_1(\mN)},
    \end{align*}
which gives Theorem \ref{Second theorem}.

\subsection{Estimate for the good function (\ref{estimate of good function}) }
For any $i\in\Z$, it follows from the definitions of $S_i$ and $T_{\epsilon}^{i_{\epsilon}}$ ((\ref{Sif(x)}) and (\ref{T epsilon,i epsilon f(x)})) that there exists $\epsilon_i > 0$ such that
     \[ S_i = S_{\epsilon_i}- T_{\epsilon_i}^i. \]
Under the assumptions of Theorem \ref{Main theorem} and Lemma \ref{reduction to lacunary sequence} $(i)$, both $(S_{\epsilon_i})_{i \in \mathbb{Z}}$ and $(T_{\epsilon_i}^i)_{i \in \mathbb{Z}}$ are of strong type $(p_0, p_0)$.
Therefore, from the above equality, $(S_i)_{i \in \mathbb{Z}}$ is also of strong type $(p_0, p_0)$.
In addition, since $g$ is positive (by Theorem \ref{nc CZ decomposition with convolution}) and the kernel $K$ is real-valued (as stated in Section \ref{Section: Two reductions}), the operator $S_i g$ is self-adjoint.
These two facts imply that for any $i\in\Z$, we can find a positive element $a\in\mN$ such that for any $i\in\Z$,
    \[
    -a \leq S_i g \leq a \quad\text{and}\quad \norm{a}_{L_{p_0}(\mN)}\lesssim \norm{g}_{L_{p_0}(\mN)}.
    \]
Taking $e_1 = \chi_{(0,\lambda]}(a)$. Then we have
    \[
    -\lambda \leq -e_1 a e_1 \leq e_1 S_i g e_1 \leq e_1 a e _1 \leq \lambda.
    \]
Finally, applying the Chebyshev inequality, the H\"older inequality, along with the boundedness of $g$ in Lemma \ref{properties of nc CZ decomposition with convolution} $(i)$, we obtain
     \begin{align*}
	 \phi(e_1^{\bot})
     \leq \lambda^{-p_0} \norm{a}_{L_{p_0}(\mN)}^{p_0}
     \leq \lambda^{-p_0} \norm{g}_{L_{\infty}(\mN)}^{p_0-1} \norm{g}_{L_1(\mN)}
     \lesssim \lambda^{-1} \norm{f}_{L_1(\mN)},
     \end{align*}
which yields (\ref{estimate of good function}).

\subsection{Estimate for the bad function (\ref{estimate of bad function})}

We now give an estimate for the bad function $b$.
First, we decompose $S_i b$ as
    \[
    S_i b = \zeta^{\bot}S_i b\zeta^{\bot} + \zeta S_i b\zeta^{\bot} + \zeta^{\bot}S_i b\zeta + \zeta S_i b\zeta,
    \]
where the projection $\zeta$ is defined as in (\ref{projection zeta_s}), and we omit its subscript as stated in Remark \ref{the value of s}.

For each $j\in\Z$, define the operator $\mathcal{T}_j$ by
    \[
    \mathcal{T}_j f(x)= \int_{\R^d} K_j(x,y) f(y) \dif y.
    \]
Then, $S_i$ admits a representation $S_i = \sum\limits_{j\geq i} \mathcal{T}_j$.
The key step is to establish the following $L_1$-estimate:
    \begin{align} \label{zeta Tj b zeta}
	\sum_{j\in\Z} \norm{\zeta \mathcal{T}_j b \zeta}_{L_1(\mN)} \lesssim \norm{f}_{L_1(\mN)}.
    \end{align}
We claim that (\ref{zeta Tj b zeta}) holds, and use it to deduce (\ref{estimate of bad function}).
The proof of (\ref{zeta Tj b zeta}) itself will be given later.

Define the projection $\eta = \chi_{(0,\lambda]}\Big(\sum\limits_{j\in\Z} |\zeta \mathcal{T}_j b \zeta| \Big)$.
It then follows that
    \begin{align*}
	-\lambda
	\leq -\eta \sum_{j\in\Z} |\zeta \mathcal{T}_j b \zeta| \eta
	\leq \eta \zeta S_i b \zeta \eta
	\leq \eta \sum_{j\in\Z} |\zeta \mathcal{T}_j b \zeta| \eta
	\leq \lambda.
     \end{align*}
Moreover, by the Chebyshev inequality and (\ref{zeta Tj b zeta}),
    \begin{align*}
	\phi(\eta^{\perp}) \leq \lambda^{-1} \sum_{j\in\Z} \norm{\zeta \mathcal{T}_j b\zeta }_{L_1(\mN)}
	\lesssim \lambda^{-1} \norm{f}_{L_1(\mN)}.
    \end{align*}
Now, let $e_2 = \zeta \wedge\eta$.
Combining these results with Lemma \ref{boundedness of 1-zeta}, we obtain
    \begin{align*}
	\norm{e_2 S_i b e_2}_{L_{\infty}(\mN)}
	\leq \norm{\eta \zeta S_i b \zeta \eta}_{L_{\infty}(\mN)}
	\leq \lambda
    \end{align*}
    and
    \begin{align*}
	\phi(e_2)
	\leq \phi(\zeta^{\perp}) + \phi(\eta^{\perp})
	\lesssim \lambda^{-1} \norm{f}_{L_1(\mN)},
    \end{align*}
which leads to (\ref{estimate of bad function}).

It remains to prove (\ref{zeta Tj b zeta}).
Recall from Theorem \ref{nc CZ decomposition with convolution} that the bad function $b$ has the decomposition:
    \[
    b = \widetilde{b}_{\text{d}} + b_{\text{d}}^0 + \widetilde{b}_{\text{off}} + b_{\text{off}}^0.
    \]
Therefore, by the triangle inequality, it suffices to verify the following two lemmas:
\begin{lemma}\label{Diagonal part zeta Tj bd zeta}
$(i)$ $\sum\limits_{j\in\Z} \big\|\zeta \mathcal{T}_j \widetilde{b}_{\text{d}} \zeta \big\|_{L_1(\mN)} \lesssim \norm{f}_{L_1(\mN)}$;
$(ii)$ $\sum\limits_{j\in\Z} \big\|\zeta \mathcal{T}_j b^0_{\text{d}} \zeta \big\|_{L_1(\mN)} \lesssim \norm{f}_{L_1(\mN)}$.
\end{lemma}

\begin{lemma}\label{Off-diagonal part zeta Tj boff zeta}
$(i)$ $\sum\limits_{j\in\Z} \big\| \zeta \mathcal{T}_j \widetilde{b}_{\text{off}} \zeta \big\|_{L_1(\mN)} \lesssim \norm{f}_{L_1(\mN)}$;
$(ii)$ $\sum\limits_{j\in\Z} \big\| \zeta \mathcal{T}_j b_{\text{off}}^0 \zeta \big\|_{L_1(\mN)} \lesssim \norm{f}_{L_1(\mN)}$.
\end{lemma}

We first introduce a lemma that will be essential for proving Lemma \ref{Diagonal part zeta Tj bd zeta} and Lemma \ref{Off-diagonal part zeta Tj boff zeta}.
For a fixed dyadic cube $Q$, we recall the kernel average
    \[
    \mathcal{K}_Q(x,y) = \frac{1}{|Q|}\int_{|v|\leq \frac{\sqrt{d}}{2}\ell(Q)} |K(x,y+v)-K(x,y)| \dif v.
    \]
When there is no risk of confusion, it is convenient to denote by $\mathcal{K}_{j,Q}$ the corresponding average for the truncated kernel $K_j$:
    \[
    \mathcal{K}_{j,Q}(x,y) := \frac{1}{|Q|}\int_{|z|\leq \frac{\sqrt{d}}{2} \ell(Q)} |K_j(x,y+z)-K_j(x,y)| \dif z .
    \]

\begin{lemma} \label{the first lemma for bad function}
Fix a dyadic cube $Q\in \mQ_{n-j}$ and $j\in\Z$.
Then, for $n\geq 3$,	
		\begin{align*}
	    \int_{\R^d}  \sup_{y\in Q} \mathcal{K}_{j,Q}(x,y) \dif x
        \lesssim  \mathcal{H}(n-2) + \mathcal{H}(n-1) + \mathcal{H}(n) + \mathcal{H}(n+1) +  2^{-n}.
		\end{align*}  	
\end{lemma}

\begin{proof}
From the definition of $K_j$ in (\ref{Kj}), we have
   	\begin{align*}
   	\mathcal{K}_{j,Q}(x,y)
   	&\leq \frac{1}{|Q|} \Big( \int_{|z|\leq \frac{\sqrt{d}}{2}\ell(Q)} |K(x,y+z) - K(x,y)| |\Phi_j(x-y)|  \dif z \\
   	&\quad + \int_{|z|\leq \frac{\sqrt{d}}{2}\ell(Q) } |K(x,y+z)| |\Phi_j(x-(y+z))-\Phi_j(x-y)|  \dif z \Big).
   	\end{align*}
Consequently,    	
    \begin{align*}
    \int_{\R^d}  \sup_{y\in Q} \mathcal{K}_{j,Q}(x,y)\dif x
    \leq A_1 + A_2,
    \end{align*}
where
    \begin{align*}
    A_1
    &:= \frac{1}{|Q|} \int_{\R^d} \sup_{y\in Q} \Big( \int_{|z|\leq \frac{\sqrt{d}}{2}\ell(Q)} |K(x,y+z) - K(x,y)| \dif z |\Phi_j(x-y)| \Big)\dif x \\
    &= \int_{\R^d} \sup_{y\in Q} \Big(\mathcal{K}_{Q}(x,y) |\Phi_j(x-y)| \Big) \dif x,
    \end{align*}
and
    \begin{align*}
    A_2:=  \frac{1}{|Q|} \int_{\R^d} \Big(\sup_{y\in Q} \int_{|z|\leq \frac{\sqrt{d}}{2}\ell(Q) } |K(x,y+z)| |\Phi_j(x-(y+z))-\Phi_j(x-y)|  \dif z \Big) \dif x.
    \end{align*}

We first estimate $A_1$.
When $y\in Q$, by $\supp(\Phi_j) \subset\big\{x\in\R^d:2^{j-1}\sqrt{d} \leq |x|\leq 2^{j+1}\sqrt{d} \big\}$ and
    \begin{align} \label{the range of z}
    \frac{\sqrt{d}}{2}\ell(Q) = 2^{-n+j-1}\sqrt{d} \leq 2^{j-4}\sqrt{d}, \quad \text{when}\ n\geq 3,
    \end{align}
 we obtain
    \begin{align*}
    |x-c(Q)| \leq |x-y|+|y-c(Q)| \leq 2^{j+1}\sqrt{d} + (\sqrt{d}/2)\ell(Q) \leq 2^{j+2}\sqrt{d}, \\
    |x-c(Q)| \geq |x-y|-|y-c(Q)| \leq 2^{j-1}\sqrt{d} - (\sqrt{d}/2)\ell(Q) \leq 2^{j-2}\sqrt{d}.
    \end{align*}
This means,
    \begin{align*}
    2^{n-2}\sqrt{d}\ell(Q) \leq |x-c(Q)| \leq 2^{n+2}\sqrt{d}\ell(Q).	
    \end{align*}
Hence,
    \begin{align*}
	A_1
	&\leq \int_{2^{n-2}\sqrt{d}\ell(Q)\leq|x-c(Q)|\leq 2^{n+2}\sqrt{d}\ell(Q)} \sup_{y\in Q} \mathcal{K}(x,y) \dif x\\
	&= \sum_{k=-2}^1 \int_{2^{n+k}\sqrt{d}\ell(Q)\leq|x-c(Q)|\leq 2^{n+k+1}\sqrt{d}\ell(Q)} \sup_{y\in Q} \mathcal{K}(x,y) \dif x \\
	&\leq \mathcal{H}(n-2) + \mathcal{H}(n-1) + \mathcal{H}(n) + \mathcal{H}(n+1).
    \end{align*}

Consider $A_2$.
We first claim that there exists a constant $C_{d}$ such that the function $x\mapsto \big(\Phi_j(x-(y+z))-\Phi_j(x-y)\big)$ is supported in the set
    \begin{align}\label{support E(j,Q)}
	 \big\{x\in\R^d: |x-c(Q)| \leq C_{d}2^{j}\big\}
	 =:E_{j,Q},
    \end{align}
whenever $y\in Q$ and $z\in B_{\frac{\sqrt{d}}{2}\ell(Q)}$.
To prove this, by the support of $\Phi_j$, when $\Phi_j(x-(y+z))-\Phi_j(x-y) \neq 0$, at least one of the following conditions holds:
    \begin{enumerate}[$(i)$]
	\item $2^{j-1}\sqrt{d} \leq |x-(y+z)| \leq 2^{j+1}\sqrt{d}$,
	\item $2^{j-1}\sqrt{d} \leq |x-y| \leq 2^{j+1}\sqrt{d}$.
    \end{enumerate}
In case $(i)$, it follows from $|z|\leq \frac{\sqrt{d}}{2}\ell(Q)$, $y\in Q$ and (\ref{the range of z}) that    \begin{align*}
    |x-c(Q)| \leq |x-(y+z)| + |z| + |y-c(Q)| \leq 2^{j+1}\sqrt{d} +   2^{j-4}\sqrt{d} + 2^{j-4}\sqrt{d}.
    \end{align*}
In case $(ii)$, similarly,
    \begin{align*}
    |x-c(Q)| \leq |x-y| + |y-c(Q)| \leq 2^{j+1}\sqrt{d} +   2^{j-4}\sqrt{d}.
    \end{align*}
Thus (\ref{support E(j,Q)}) holds.

Moreover, when $\Phi_j(x-(y+z))-\Phi_j(x-y) \neq 0$, we have another useful estimate:
    \begin{align} \label{|x-(y+z)|=2j}
	|x-(y+z)| \approx 2^j.
    \end{align}
Indeed, this is immediate in case $(i)$.
And in case $(ii)$, using $|z|\leq \frac{\sqrt{d}}{2}\ell(Q)$ and (\ref{the range of z}) again, we have
   \begin{align*}
   |x-(y+z)| \leq |x-y| + |z| \leq 2^{j+1}\sqrt{d} + 2^{j-4}\sqrt{d}, \\
   |x-(y+z)| \geq |x-y| - |z| \geq 2^{j-1}\sqrt{d} - 2^{j-4}\sqrt{d},
   \end{align*}
which also implies $|x-(y+z)|\approx 2^{j}$, leading to (\ref{|x-(y+z)|=2j}).

Now, we can bound $A_2$.
By the size condition (\ref{Size condition}), the mean-value theorem, the bound $|z|\leq \frac{\sqrt{d}}{2}\ell(Q) = 2^{-n+j-1}\sqrt{d}$, together with (\ref{support E(j,Q)}) and (\ref{|x-(y+z)|=2j}),
    \begin{align*}
    A_2	
    &\lesssim \int_{E_{j,Q}} \Big(\sup_{y\in Q} \frac{1}{|Q|} \int_{|z|\leq \frac{\sqrt{d}}{2}\ell(Q)}  \frac{ |z|2^{-j}}{|x-(y+z)|^d} \dif z \Big) \dif x \\
    &\lesssim 2^{-n-jd}\int_{E_{j,Q}} \Big(\sup_{y\in Q} \frac{1}{|Q|} \int_{|z|\leq \frac{\sqrt{d}}{2}\ell(Q)} 1 \dif z \Big) \dif x \\
    &\lesssim 2^{-n}.
    \end{align*}
This completes the proof.
\end{proof}

\subsubsection{The proof of Lemma \ref{Diagonal part zeta Tj bd zeta}}

$(i)$
We first observe that, for any $j\in\Z$,
    \begin{align} \label{a claim to reduce the sum}
	\zeta \mathcal{T}_j \widetilde{b}_{\text{d},n-j}\zeta
	= \zeta \mathcal{T}_j b_{\text{d},n-j}^0\zeta = 0, \quad \text{when} \ n \leq 3.
    \end{align}
Indeed, for a fixed $x\in\R^d$, Lemma \ref{properties of nc CZ decomposition with convolution} $(ii)$ and Remark \ref{the value of s} ($s = [50\sqrt{d}]$) implies
    \begin{align*}
	&\zeta(x) \mathcal{T}_j \widetilde{b}_{\text{d},n-j}(x) \zeta(x) \\
	&\quad = \int_{\R^d} K_j(x,y) \Big(\zeta(x)\widetilde{b}_{\text{d},n-j}(y)\zeta(x) \Big) \chi_{\big\{|x-y|> ([50\sqrt{d}]-\sqrt{d})2^{-n+j} \big\}}(y) \dif y.
    \end{align*}
Since $\supp (\Phi_j) \subseteq \big\{ x\in\R^d: 2^{j-1}\sqrt{d} \leq |x| \leq 2^{j+1}\sqrt{d} \big\}$,  we get
    \[
    K_j(x,y) = K(x,y)\Phi_j(x-y) = 0,
    \]
when
    \[
    |x-y|
    \geq([50\sqrt{d}]-\sqrt{d})2^{-n+j}
    \geq ([50\sqrt{d}]-\sqrt{d}) 2^{j-3}
    \geq 2^{j+1}\sqrt{d},
    \]
where the second inequality uses $n\leq 3$.
This shows that $\zeta \mathcal{T}_j \widetilde{b}_{\text{d},n-j}\zeta = 0$ when $n\leq 3$.
The same argument works for $\zeta \mathcal{T}_j b_{\text{d},n-j}^0 \zeta$, giving rise to (\ref{a claim to reduce the sum}).

We begin with estimating $\widetilde{b}_{\text{d}}$.
By Remark \ref{sums of bd from 1 to infty} and (\ref{a claim to reduce the sum}), it suffices to establish 	
    \begin{align*}
    \sum_{n\geq 4} \sum_{j\leq n-1} \big\|\zeta \mathcal{T}_j \widetilde{b}_{\text{d},n-j} \zeta \big\|_{L_1(\mN)} \lesssim \norm{f}_{L_1(\mN)}.
    \end{align*}

It follows from the definition of $\widetilde b_{\text{d},n}$ in Theorem \ref{nc CZ decomposition with convolution}, the cancellation condition in Lemma \ref{properties of nc CZ decomposition with convolution} $(ii)$, and the definition of $\widetilde{b}_{\text{d},{n-j}}^Q$ in Theorem \ref{nc CZ decomposition with convolution} that
    \begin{align*}
	\mathcal{T}_j \widetilde b_{\text{d},n-j} (x)
	&= \sum_{Q\in\mQ_{n-j}} \int_{\R^d} \big(K_j(x,y)-K_j(x,c(Q)\big) \widetilde{b}_{\text{d},{n-j}}^Q(y) \dif y \\
	&= \sum_{Q\in\mQ_{n-j}} \int_{\R^d} \int_{\R^d}\big(K_j(x,y)-K_j(x,c(Q)\big) b_{\text{d},n-j}^Q(u)h_{\frac{\sqrt{d}}{2}\ell(Q)}(y-u) \dif y \dif u.
    \end{align*}
Then, by $\supp (h_{\frac{\sqrt{d}}{2}\ell(Q)}) \subseteq B_{\frac{\sqrt{d}}{2}\ell(Q)}$, $\supp (b^Q_{\text{d},n-j}) \subseteq Q$, and using the Minkowski inequality,
    \begin{align*}
	&\big\|\zeta \mathcal{T}_j \widetilde b_{\text{d},n-j} \zeta \big\|_{L_1(\mN)} \\
	&\lesssim \sum_{Q\in\mQ_{n-j}} \int_{\R^d} \Big( \frac{\norm{h}_{\infty}}{\ell(Q)^d} \int_{Q}  \int_{ B_{\frac{\sqrt{d}}{2}\ell(Q)}(u) }  |K_j(x,y)-K_j(x,c(Q))|  \|b_{\text{d},n-j}^Q(u)\|_{L_1(\mM)} \dif y \dif u \Big) \dif x \\
	&\lesssim \sum_{Q\in\mQ_{n-j}} \int_{Q} \Big( \int_{\R^d} \frac{1}{|Q|}\int_{B_{\frac{\sqrt{d}}{2}\ell(Q)}(u)}  |K_j(x,y)-K_j(x,c(Q))|  \dif y \dif x \Big) \|b_{\text{d},n-j}^Q(u)\|_{L_1(\mM)}\dif u.
    \end{align*}
Fix a cube $Q\in\mQ_{n-j}$.
For $y\in B_{\frac{\sqrt{d}}{2}\ell(Q)}(u)$, we have
    \begin{align*}
    |y-c(Q)|
    \leq |y-u| + |u-c(Q)|
    \leq \frac{\sqrt{d}}{2}\ell(Q) + \frac{\sqrt{d}}{2}\ell(Q)
    \leq \sqrt{d}\ell(Q)
    \end{align*}
whenever $u\in Q$.
Hence, by letting $z=y-c(Q)$,
    \begin{align}\label{z< sqrt d l(Q)}
    & \frac{1}{|Q|}\int_{B_{\frac{\sqrt{d}}{2}\ell(Q)}(u)} |K_j(x,y)-K_j(x,c(Q))|  \dif y  \notag\\
	&\quad \lesssim  \frac{1}{|Q|}\int_{|z| \leq \sqrt{d}\ell(Q)}  |K_j(x,z+c(Q))-K_j(x,c(Q))| \dif z .
     \end{align}
Now, let the cube $\widehat Q$ be the unique father cube of $Q$ in $\mQ_{n-j-1}$.
This implies
    \begin{align*}
	 (\ref{z< sqrt d l(Q)})
	 &= \frac{2^d}{|\widehat Q|}  \int_{|z| \leq \frac{\sqrt{d}}{2}\ell(\widehat Q)} |K_j(x,z+c(Q))-K_j(x,c(Q))| \dif z
	 = \mathcal{K}_{j,\widehat Q}(x, c(Q)) \\
	 &\leq \sup_{y\in \widehat Q}\mathcal{K}_{j,\widehat Q}(x, y),
    \end{align*}
which leads to
    \begin{align*}
	 &\int_{\R^d} \frac{1}{|Q|}\int_{B_{\frac{\sqrt{d}}{2}\ell(Q)}(u)}  |K_j(x,y)-K_j(x,c(Q))|  \dif y \dif x \\
	 &\quad\lesssim \int_{\R^d} \sup_{y\in \widehat Q}\mathcal{K}_{j,\widehat Q}(x, y) \dif x \\
	 &\quad\lesssim \mathcal{H}(n-3) + \mathcal{H}(n-2) + \mathcal{H}(n-1) + \mathcal{H}(n) + 2^{-n+1},
    \end{align*}
where the last bound comes from Lemma \ref{the first lemma for bad function} with $n$ replaced by $n-1$.
Finally, applying the $L_1$-integral condition (\ref{L1 integral condition}) and the boundedness of $b_{\text{d},n}^Q$ in Lemma \ref{properties of nc CZ decomposition with convolution} $(ii)$, we obtain
    \begin{align*}
	\sum_{n\geq 4} \sum_{j\leq n-1} \big\|\zeta \mathcal{T}_j \widetilde{b}_{\text{d}} \zeta \big\|_{L_1(\mN)}
	&\lesssim  \sum_{n\geq 4} \sum_{j\leq n-1} \sum_{Q\in\mQ_{n-j}} \int_{Q}  \big\|b_{\text{d},n-j}^Q(u) \big\|_{L_1(\mM)} \dif u\\
	&\quad\quad \lbracket{\mathcal{H}(n-3) + \mathcal{H}(n-2) + \mathcal{H}(n-1) + \mathcal{H}(n) + 2^{-n+1}}  \\
	&\ \lesssim \sum_{n\geq 4} \big(4\mathcal{H}(n-3)+  2^{-n+1}\big) \sum_{m\geq 1} \sum_{Q\in\mQ_m} \big\|b_{\text{d},m}^Q\big\|_{L_1(\mN)} \\
	&\ \lesssim \norm{f}_1.
    \end{align*}

$(ii)$	
Again by Remark \ref{sums of bd from 1 to infty} and (\ref{a claim to reduce the sum}), it suffices to prove:
    \begin{align*}
	\sum_{n\geq 4} \sum_{j\leq n-1} \norm{\zeta \mathcal{T}_j b_{\text{d},n-j}^0 \zeta}_{L_1(\mN)}
	\lesssim \norm{f}_{L_1(\mN)}.
    \end{align*}

Fix a cube $Q\in\mQ_{n-j}$.
From the definition of $\widetilde{b}_{\text{d},n-j}^Q$ in Theorem \ref{nc CZ decomposition with convolution},
    \begin{align*}
    \mathcal{T}_j \widetilde{b}_{\text{d},n-j}^Q(x)
    &= \int_{\R^d} K_j(x,y) \Big(\int_{\R^d} b_{\text{d},n-j}^Q(y-z) h_{\frac{\sqrt{d}}{2}\ell(Q)}(z)\dif z \Big) \dif y \\
    &= \int_{\R^d} \Big(\int_{\R^d} K_j(x,y+z) h_{\frac{\sqrt{d}}{2}\ell(Q)}(z)\dif z\Big) b_{\text{d},n-j}^Q(y) \dif y.
    \end{align*}
Moreover, by $\int_{\R^d}h_{\frac{\sqrt{d}}{2}\ell(Q)}=1$,
    \begin{align*}
	\mathcal{T}_j b_{\text{d},n-j}^Q(x)
	= \int_{\R^d} \Big(\int_{\R^d} K_j(x,y) h_{\frac{\sqrt{d}}{2}\ell(Q)}(z) \dif z \Big)b_{\text{d},n-j}^Q(y) \dif y.
    \end{align*}
Noting again the supports of $h_{\frac{\sqrt{d}}{2}\ell(Q)}$, $b_{\text{d},n-j}^Q$ and using the Minkowski inequality, we get
    \begin{align*}
	&\norm{\zeta \mathcal{T}_j b_{\text{d},n-j}^0 \zeta}_{L_1(\mN)} \\
	&\ \lesssim  \sum_{Q\in\mQ_{n-j}} \int_{Q} \int_{\R^d}\Big( \frac{1}{|Q|}\int_{|z|\leq \frac{\sqrt{d}}{2}\ell(Q)}  |K_j(x,y+z)-K_j(x,y)| \dif z \Big) \dif x \|b_{\text{d},n-j}^Q(y)\|_{L_1(\mM)} \dif y \\
	&\ \lesssim \big(\mathcal{H}(n-2) + \mathcal{H}(n-1) + \mathcal{H}(n) + \mathcal{H}(n+1) + 2^{-n}\big) \sum_{Q\in\mQ_{n-j}}\|b_{\text{d},n-j}^Q\|_{L_1(\mN)},
    \end{align*}
where the last inequality follows from Lemma \ref{the first lemma for bad function}.
Finally, an application of the $L_1$-integral condition \eqref{L1 integral condition}, and the boundedness of $b_{\mathrm{d},n}^Q$ in Lemma \ref{properties of nc CZ decomposition with convolution} $(ii)$, yield
    \begin{align*}
	\sum_{n\geq 4} \sum_{j\leq n-1} \norm{\zeta \mathcal{T}_j b_{\text{d},n-j}^0 \zeta}_{L_1(\mN)}
	&\lesssim \sum_{n\geq 4} \lbracket{4\mathcal{H}(n-2) + 2^{-n}} \sum_{m\geq 1} \sum_{Q\in\mQ_m} \|b_{\text{d},m}^Q\|_{L_1(\mN)} \\
	&\lesssim  \norm{f}_{L_1(\mN)},
    \end{align*}
which completes the argument.

\subsubsection{The proof of Lemma \ref{Off-diagonal part zeta Tj boff zeta}}

$(i)$
As in the proof for the diagonal part $b_{\text{d}}$, we state the following equality, whose proof is similar to that of (\ref{a claim to reduce the sum}):
for any $j\in\Z$,
    \begin{align}\label{a claim to reduce the sum 2}
    \zeta \mathcal{T}_j \widetilde{b}_{\text{off},n-j}\zeta
    =\zeta \mathcal{T}_j b_{\text{off},n-j}^0 \zeta =0, \quad\text{when}\ n\leq 3.
    \end{align}
Together with Remark \ref{sums of bd from 1 to infty}, it suffices to prove
    \begin{align*}
    \sum_{n\geq4} \sum_{j\leq n-1} \big\|\zeta \mathcal{T}_j \widetilde{b}_{\text{off},n-j} \zeta \big\|_{L_1(\mN)}
    \lesssim \norm{f}_{L_1(\mN)}.
    \end{align*}

For a fixed $x\in\R^d$, the definition of $\widetilde{b}_{\text{off},n-j}$ in Theorem \ref{nc CZ decomposition with convolution} and the cancellation condition in Lemma \ref{properties of nc CZ decomposition with convolution} $(iv)$ give
    \begin{align*}
    \mathcal{T}_j \widetilde{b}_{\text{off},n-j} (x)
    = \sum_{Q\in\mQ_{n-j}}\int_{\R^d} \big(K_j(x,y)-K_j(x,c(Q)) \big) \widetilde{b}_{\text{off},n-j}^Q(y) \dif y.
    \end{align*}
Next, recall the definition of $\widetilde{b}_{\text{off},n}^Q$ in Theorem \ref{nc CZ decomposition with convolution}: for a fixed cube  $Q\in\mQ_n$, $\widetilde{b}_{\text{off},n}^Q = b_{\text{off},n}^Q*h_{\frac{\sqrt{d}}{2}\ell(Q)}$, where $b_{\text{off},n}^Q = p_Q(f-f_n)q_Q\chi_Q + q_Q(f-f_n)p_Q\chi_Q$.
Accordingly, we can  majorize $\sum\limits_{n\geq4} \sum\limits_{j\leq n-1} \big\|\zeta \mathcal{T}_j \widetilde{b}_{\text{off},n-j} \zeta \big\|_{L_1(\mN)}$ by the sum of the following four terms:
    \begin{align*}
	F_i
	:= \sum_{n\geq 4} \sum_{j\leq n-1}\sum_{Q\in\mQ_{n-j}} \int_{\R^d}  \Big\|\int_{\R^d} \big(K_j(x,y)-K_j(x,c(Q))\big) \widetilde{b}_{\text{off},n-j}^{Q,i}(y)  \dif y\Big\|_{L_1(\mM)} \dif x,
    \end{align*}
for $i=1,2,3,4$, where
    \begin{align*}
	&\widetilde{b}_{\text{off},n-j}^{Q,1}  := p_Q \big((f\chi_Q) * h_{\frac{\sqrt{d}}{2}\ell(Q)}\big) q_Q, \quad
	\widetilde{b}_{\text{off},n-j}^{Q,2}  := p_Q \big((f_Q\chi_Q) * h_{\frac{\sqrt{d}}{2}\ell(Q)}\big) q_Q, \\
	&\widetilde{b}_{\text{off},n-j}^{Q,3}  := q_Q \big((f\chi_Q) * h_{\frac{\sqrt{d}}{2}\ell(Q)}\big) p_Q, \quad
	\widetilde{b}_{\text{off},n-j}^{Q,4}  := q_Q \big((f_Q\chi_Q) * h_{\frac{\sqrt{d}}{2}\ell(Q)}\big) p_Q.
    \end{align*}
We provide a detailed proof only for $F_1$;  analogous arguments apply to $F_2$--$F_4$ under the fact that $\phi(f_n p_n) = \phi(f p_n)$.


Define
    \begin{align*}
    \mathcal{L}_j(x,u)
    := \int_{|u-y|\leq \frac{\sqrt{d}}{2}\ell(Q)} (K_j(x,y)-K_j(x,c(Q))) h_{\frac{\sqrt{d}}{2}\ell(Q)}(y-u)\dif y,	
    \end{align*}
and decompose it into positive and negative parts:
    \begin{align*}
    \mathcal{L}_j(x,u) = \mathcal{L}_j^+(x,u) - \mathcal{L}_j^-(x,u).
    \end{align*}
Hence, for a fixed cube $Q\in\mQ_{n-j}$, using the supports of $h_{\frac{\sqrt{d}}{2}\ell(Q)}$ and $\chi_Q$, the noncommutative $L_1$-norm inside $F_1$ equals
    \begin{align*}
    \Big\|\int_Q \mathcal{L}_j(x,u) p_Q f(u) q_Q\dif u \Big\|_{L_1(\mM)}.
    \end{align*}
By the triangle inequality, it is enough to treat the positive part, i.e.
    \begin{align} \label{the nc L1 norm in F_1}
    \Big\|\int_Q \mathcal{L}_j^+(x,u) p_Q f(u) q_Q\dif u \Big\|_{L_1(\mM)}.
    \end{align}


Let $\widehat Q$ be the unique father cube of $Q$ in $\mQ_{n-j-1}$.
As in the proof of Lemma \ref{Diagonal part zeta Tj bd zeta} $(i)$, when $u\in Q$, we can estimate $\mathcal{L}_j^+(x,u)$ as follows:
   \begin{align} \label{estimate for |Lj(x,u)|}
   	|\mathcal{L}_j^+(x,u)|
   	&\lesssim \frac{1}{|Q|} \int_{|u-y|\leq \frac{\sqrt{d}}{2}\ell(Q)} |K_j(x,y)-K_j(x,c(Q))|\dif y \notag\\
   	&\lesssim \frac{1}{|\widehat Q|} \int_{|z|\leq \frac{\sqrt{d}}{2}\ell(\widehat Q)} |K_j(x,z+c(Q))-K_j(x,c(Q))|\dif z \notag\\
   	&\leq \sup_{y\in \widehat Q} \mathcal{K}_{j,\widehat Q}(x, y).
    \end{align}

By Lemma \ref{also a useful lemma for bad function} in Appendix \ref{Appendix: A Holder inequality in column spaces}, there exists a contraction $m(x)$ such that
    \begin{align*}
    &\int_Q \mathcal{L}_j^+(x,u) p_Q f(u) q_Q\dif u \\
    &\quad=\Big( \int_Q \mathcal{L}^+_j(x,u) p_Q f(u) p_Q \dif u \Big)^{\frac{1}{2}} m(x) \Big( \int_Q \mathcal{L}_j^+(x,u) q_Q f(u) q_Q d\dif u \Big)^{\frac{1}{2}}\\
    &\quad=\Big( \int_Q \mathcal{L}^+_j(x,u) p_Q f(u) p_Q \dif u \Big)^{\frac{1}{2}} p_Qm(x) \Big( \int_Q \mathcal{L}^+_j(x,u) q_Q f(u) q_Q d\dif u \Big)^{\frac{1}{2}}.
    \end{align*}
Here we used that $(p_Q X p_Q)^{\frac{1}{2}} = (p_Q X p_Q)^{\frac{1}{2}} p_Q$ when $X$ is positive.
Together it with the H\"older inequality and (\ref{estimate for |Lj(x,u)|}), we can bound (\ref{the nc L1 norm in F_1}) by
    \begin{align*}	
    &\Big\| \Big( \int_Q \mathcal{L}_j^+(x,u) p_Q f(u) p_Q \dif u \Big)^{\frac{1}{2}} \Big\|_{L_2(\mM)} \Big\| p_Q m(x)\Big( \int_Q \mathcal{L}_j^+(x,u) q_Q f(u) q_Q \dif u \Big)^{\frac{1}{2}} \Big\|_{L_2(\mM)} \\
    &\lesssim \Big(\sup_{y\in \widehat Q} \mathcal{K}_{j,\widehat Q}(x, y)\Big) \tau\Big( \int_Q p_Q f(u) p_Q \dif u \Big)^{\frac{1}{2}} \tau\Big( p_Q m(x)\big( \int_Q q_Q f(u) q_Q \dif u \big) m(x)^* p_Q\Big)^{\frac{1}{2}}.
    \end{align*}
On the one hand, it is easy to see
    \begin{align*}
    \tau\Big( \int_Q p_Q f(u) p_Q \dif u \Big)
    \leq \phi\big(f p_Q \chi_Q \big).
    \end{align*}
On the other hand, by the definition of $f_Q$, (\ref{q_Q f_Q q_Q < lambda q_Q}) and $|m(x)|^2\leq 1$,
    \begin{align*}
    \tau\Big( p_Q m(x)\big( \int_Q q_Q f(u) q_Q \dif u \big) m(x)^* p_Q\Big)
    &= \tau\big(|Q| p_Q m(x) q_Q f_Q q_Q  m(x)^* p_Q \big) \\
    &\leq \lambda |Q| \tau\big( p_Q m(x) q_Q  m(x)^* p_Q \big) \\
    &\leq \lambda |Q|\tau(p_Q).
    \end{align*}
Hence,
    \begin{align*}
    (\ref{the nc L1 norm in F_1})
    \lesssim \Big(\sup_{y\in \widehat Q} \mathcal{K}_{j,\widehat Q}(x, y)\Big) \phi\big(f p_Q \chi_Q \big)^{\frac{1}{2}} \big( \lambda |Q|\tau(p_Q) \big)^{\frac{1}{2}}.
    \end{align*}
This estimate, combined with Lemma \ref{the first lemma for bad function} and the  $L_1$-integral condition (\ref{L1 integral condition}), implies
    \begin{align*}
    F_1
    &\lesssim \sum_{n\geq 4} \sum_{j\leq n-1}\sum_{Q\in\mQ_{n-j}} \Big( \int_{\R^d} \sup_{y\in \widehat Q} \mathcal{K}_{j,\widehat Q}(x, y) \dif x \Big) \phi\big(f p_Q \chi_Q \big)^{\frac{1}{2}} \big( \lambda |Q|\tau(p_Q) \big)^{\frac{1}{2}} \\
    &\lesssim \sum_{n\geq 4}  \Big(\mathcal{H}(n-3) + \mathcal{H}(n-3) + \mathcal{H}(n-1) + \mathcal{H}(n) + 2^{-n+1}\Big) \\
    &\quad\quad \sum_{j\leq n-1}\sum_{Q\in\mQ_{n-j}}\phi\big(f p_Q \chi_Q \big)^{\frac{1}{2}} \big( \lambda |Q|\tau(p_Q) \big)^{\frac{1}{2}} \\
    &\lesssim \sum_{m\geq 1}\sum_{Q\in\mQ_{m}}\phi\big(f p_Q \chi_Q \big)^{\frac{1}{2}} \big( \lambda |Q|\tau(p_Q) \big)^{\frac{1}{2}}.
    \end{align*}
Finally, we prove
   \begin{align} \label{lesssim ||f||1}
   \sum_{m\geq 1}\sum_{Q\in\mQ_{m}}\phi\big(f p_Q \chi_Q \big)^{\frac{1}{2}} \big( \lambda |Q|\tau(p_Q) \big)^{\frac{1}{2}}
   \lesssim \| f\|_{L_1(\mN)},
   \end{align}
which gives the desired bound for $F_1$.
Indeed, by applying the Cauchy-Schwarz inequality twice, we estimate the left-hand side of (\ref{lesssim ||f||1}) as
    \begin{align*}	
    &\sum_{n\in\Z} \Big( \sum_{Q\in\mQ_n} \phi(fp_Q\chi_Q)\Big)^{\frac{1}{2}} \Big( \sum_{Q\in\mQ_n} \tau (p_Q) (|Q|\lambda\Big)^{\frac{1}{2}} \\
	&\quad \leq \Big( \sum_{n\in\Z} \sum_{Q\in\mQ_n} \phi(fp_Q\chi_Q)\Big)^{\frac{1}{2}} \Big( \sum_{n\in\Z} \sum_{Q\in\mQ_n} \tau (p_Q) (|Q|\lambda) \Big)^{\frac{1}{2}} \\
	&\quad=  \Big( \sum_{n\geq 1} \phi(fp_n) \Big)^\frac{1}{2} \Big( \sum_{n\geq 1} \lambda\phi (p_n) \Big)^\frac{1}{2}.
    \end{align*}
Then, it follows from (\ref{sum pn=1-q}) that $\sum\limits_{n\geq 1}\phi(p_n) = \phi(1-q)$ and $\sum\limits_{n\geq 1}\phi(fp_n) = \phi\big(f(1-q)\big)$.
Combined with Lemma \ref{cuculescu's construction} $(iv)$, this yields
    \begin{align*}
    \sum_{m\geq 1}\sum_{Q\in\mQ_{m}}\phi\big(f p_Q \chi_Q \big)^{\frac{1}{2}} \big( \lambda |Q|\tau(p_Q) \big)^{\frac{1}{2}}	
    \leq \Big( \lambda\phi (1-q) \Big)^\frac{1}{2}  \Big( \phi(f(1-q)) \Big)^\frac{1}{2}
    \lesssim \|f\|_{L_1(\mN)},
    \end{align*}
which proves (\ref{lesssim ||f||1}) and completes the estimate for $F_1$.

$(ii)$
We now consider $b_{\text{off}}^0$.
By (\ref{a claim to reduce the sum 2}) and Remark \ref{sums of bd from 1 to infty}, it suffices to prove
    \begin{align*}
    \sum_{n\geq 4} \sum_{j\leq n-1} \big\| \zeta \mathcal{T}_j b_{\text{off},n-j}^0 \zeta \big\|_{L_1(\mN)}
    \lesssim \norm{f}_{L_1(\mN)}.
    \end{align*}

For a fixed $x\in\R^d$, from the definition of $b_{\text{off},n-j}^0$ in Theorem \ref{nc CZ decomposition with convolution}, a direct calculation implies that $\mathcal{T}_j b_{\text{off},n-j}^0(x)$ equals to
    \begin{align*}
	& \sum_{Q\in\mQ_{n-j}} \int_{\R^d} K_j(x,y)(b_{\text{off},n}^Q-\widetilde{b}_{\text{off},n}^Q)(y) \dif y \\
	&= \sum_{Q\in\mQ_{n-j}} \int_{Q} \int_{|z|\leq \frac{\sqrt{d}}{2}\ell(Q)} \big(K_j(x,y)-K_j(x,y+z) \big) h_{\frac{\sqrt{d}}{2}\ell(Q)}(z) b^Q_{\text{off},n-j}(y) \dif z\dif y.
    \end{align*}
Recall the definition of $b_{\text{off},n}^Q$ in Theorem \ref{nc CZ decomposition with convolution}: for a fixed cube  $Q\in\mQ_n$, $b_{\text{off},n}^Q = p_Q(f-f_n)q_Q\chi_Q + q_Q(f-f_n)p_Q\chi_Q$.
Again, $\sum\limits_{n\geq 4} \sum\limits_{j\leq n-1} \big\|\zeta \mathcal{T}_j b_{\text{off},n-j}^0 \zeta \big\|_{L_1(\mN)}$ can be majorized by the sum of the following four terms:
    \begin{align*}
	G_i:=
	& \sum_{n\geq 4} \sum_{j\leq n-1}\sum_{Q\in\mQ_{n-j}} \\
	&\int_{\R^d} \Big\| \int_{Q} \int_{|z|\leq \frac{\sqrt{d}}{2}\ell(Q)} \big(K_j(x,y+z)-K_j(x,y)\big) b_{\text{off},n-j}^{Q,i} (y,z) \dif z \dif y \Big\|_{L_1(\mM)}  \dif x,
    \end{align*}
for $i=1,2,3,4$, where
    \begin{align*}
	&b_{\text{off},n-j}^{Q,1} (y,z) = p_Q (f\chi_Q)(y) h_{\frac{\sqrt{d}}{2}\ell(Q)}(z) q_Q, \quad
	b_{\text{off},n-j}^{Q,2} (y,z) = p_Q (f_{Q}\chi_Q)(y) h_{\frac{\sqrt{d}}{2}\ell(Q)}(z) q_Q, \\
	&b_{\text{off},n-j}^{Q,3} (y,z) = q_Q (f\chi_Q)(y) h_{\frac{\sqrt{d}}{2}\ell(Q)}(z) p_Q, \quad
	b_{\text{off},n-j}^{Q,4} (y,z) = q_Q (f_{Q}\chi_Q )(y) h_{\frac{\sqrt{d}}{2}\ell(Q)}(z) p_Q.
    \end{align*}
We estimate only $G_1$; the other terms are similar by using $\phi(f_n p_n) = \phi(f p_n)$.

Fix a cube $Q\in\mQ_{n-j}$.
By the supports of $h_{\frac{\sqrt{d}}{2}\ell(Q)}$ and $\chi_Q$, the noncommutative $L_1$-norm inside $G_1$ equals
    \begin{align*}
    \Big\|\int_Q \mathcal{J}_j(x,y) p_Q f(y) q_Q\dif y \Big\|_{L_1(\mM)},
    \end{align*}
where
    \begin{align*}
    \mathcal{J}_j(x,y)
    := \int_{|z|\leq \frac{\sqrt{d}}{2}\ell(Q)} \big( K_j(x,y+z)-K_j(x,y) \big) h_{\frac{\sqrt{d}}{2}\ell(Q)}(z)\dif z.	
    \end{align*}
We also write $\mathcal{J}_j(x,y) = \mathcal{J}^+_j(x,y) - \mathcal{J}^-_j(x,y)$ and consider only $\mathcal{J}^+_j(x,y)$, i.e. 
    \begin{align}\label{the nc L1 norm in G_1}
    \Big\|\int_Q \mathcal{J}_j^+(x,y) p_Q f(y) q_Q\dif y \Big\|_{L_1(\mM)},
    \end{align}
Immediately, when $y\in Q$, it follows from the definition of $h_{\frac{\sqrt{d}}{2}\ell(Q)}$ that
    \begin{align*}
    |\mathcal{J}^+_j(x,y)|
    \lesssim \frac{1}{|Q|}\int_{|z|\leq \frac{\sqrt{d}}{2}\ell(Q)} |K_j(x,y+z)-K_j(x,y)| \dif z
    \leq \sup_{y\in Q} \mathcal{K}_{j,Q}(x,y).
    \end{align*}
Now, following the same reasoning as for $F_1$ in Lemma \ref{Off-diagonal part zeta Tj boff zeta} $(i)$, we bound (\ref{the nc L1 norm in G_1}) by
    \begin{align*}	
    &\Big\| \Big( \int_Q \mathcal{J}_j^+(x,y) p_Q f(y) p_Q \dif y \Big)^{\frac{1}{2}} \Big\|_{L_2(\mM)} \Big\| p_Q m'(x)\Big( \int_Q \mathcal{J}_j^+(x,y) q_Q f(y) q_Q \dif y \Big)^{\frac{1}{2}} \Big\|_{L_2(\mM)} \\
    &\lesssim \Big(\sup_{y\in Q} \mathcal{K}_{j,Q}(x, y)\Big) \tau\Big( \int_Q p_Q f(y) p_Q \dif y \Big)^{\frac{1}{2}} \tau\Big( p_Q m'(x) \big( \int_Q q_Q f(u) q_Q \dif u \big) (m'(x))^* p_Q\Big)^{\frac{1}{2}},
    \end{align*}
where $m'(x)$ is the contraction provided by Lemma \ref{also a useful lemma for bad function} in Appendix \ref{Appendix: A Holder inequality in column spaces}.
Finally, the remaining part of the proof follows the same method as used for $F_1$.
Hence, we complete the estimate for $G_1$ and the proof of Lemma \ref{Off-diagonal part zeta Tj boff zeta}.


\appendix
\section{Noncommutative $L_p$ spaces}\label{Appendix: Noncommutative Lp spaces and maximal norms}
This part consists of three subsections: noncommutative $L_p$ spaces, noncommutative maximal norms, and noncommutative square functions.
All these content are standard in noncommutative harmonic analysis, we refer to see \cite{Cuculescu1971Martingales, Junge2002Doob, Parcet2009Pseudo-localization, Pisier1998noncommutative, PisierQuanhua2003} for more details.

\subsection{Noncommutative $L_p$ spaces}
Let $\mM$ be a von Neumann algebra equipped with an $n. s. f.$ trace $\tau$.
Let $\mM_+$ be the positive part of $\mM$ and $\mS_{+}$ be the set of all  $x\in\mM_+$ with $\tau(s(x))<\infty$, where $s(x)$ is the support projection of $x$.
Let  $\mS$ be the linear span of $\mS_+$.
Then, $\mS$ is a $w^*$-dense $*$-subalgebra of $\mM$.

Let $0<p<\infty$.
For any $x\in\mS$, the operator $|x|^p \in \mS_+$ (meaning $\tau(|x|^p) <\infty$) , where $|x|=(x^* x)^{\frac{1}{2}}$ is the modulus of $x$.
We define
    \begin{align*}
	\norm{x}_{L_p(\mM)} = \big(\tau(|x|^p)\big)^{\frac{1}{p}},\quad x\in\mS.
    \end{align*}
The noncommutative $L_p$ space associated with $(\mM, \tau)$, denoted by $L_p(\mM, \tau)$ or simply $L_p(\mM)$, is defined as the completion of the space $(\mS, \norm{\cdot}_{L_p(\mM)})$.
For convenience, we set $L_{\infty}(\mM) = \mM$ equipped with the operator norm $\norm{\cdot}_{\mM}$.
We refer the reader to \cite{PisierQuanhua2003}  for more information regarding noncommutative $L_p$ spaces.


\subsection{Noncommutative maximal norms}

The generalization of the maximal function to noncommutative spaces faces a challenge: operators in von Neumann algebras are not directly comparable.
Remarkably, Pisier \cite{Pisier1998noncommutative} and Junge \cite{Junge2002Doob} overcame this difficulty by directly defining the noncommutative maximal norm.

    \begin{definition}
	For $1\leq p\leq \infty$, we define $L_p(\mM;\ell_{\infty})$ as the space of all sequences $x= (x_n)_{n\in \Z}$ in $L_p(\mM)$ which admit a factorization of the following form: there exist $a$, $b\in L_{2p}(\mM)$ and a bounded sequence $y= (y_n)_{n\in\Z}$ in $L_{\infty}(\mM)$ such that
	\begin{align*}
		x_n = a y_n b, \quad n\in\Z.
	\end{align*}
	The norm of $x \in L_p(\mM;\ell_{\infty})$ is defined as
	\begin{align*}
		\norm{x}_{L_p(\mM;\ell_{\infty})} = \inf_{x_n = a y_n b} \big\{ \norm{a}_{L_{2p}(\mM)}  \sup_{n\in\Z} \norm{y_n}_{L_{\infty}(\mM)}    \norm{b}_{L_{2p}(\mM)} \big\},
	\end{align*}
	where the infimum is taken over all factorizations of $x$ as above.
    \end{definition}
One can check that $L_p(\mM;\ell_{\infty})$ is a Banach space equipped with the noncommutative maximal norm $\norm{\cdot}_{L_p(\mM;\ell_{\infty})}$.
A positive sequence $x=(x_n)_{n\in\Z}\in L_p(\mM;\ell_{\infty})$ if and only if there exists a positive element $a\in L_p(\mM)$ such that
    \begin{align*}
	0 < x_n \leq a, \quad \forall n\in\Z.
    \end{align*}
Moreover, its norm is given by
    \begin{align*}
	\norm{x}_{L_p(\mM;\ell_{\infty})}
	= \inf\big\{\norm{a}_{L_p(\mM)}: 0 < x_n \leq a,\ \forall n\in\Z\big\}.
    \end{align*}
Similarly, when $x= (x_n)_{n\in\Z}$ is a sequence of self-adjoint operators in $L_p(\mM)$, $x \in L_p(\mM;\ell_{\infty})$ if and only if there exists a positive element $a\in L_p(\mM)$ such that
    \begin{align*}
	-a \leq x_n \leq a,\quad \forall n\in\Z.
    \end{align*}
In this case,
    \begin{align*}
	\norm{x}_{L_p(\mM;\ell_{\infty})}
	= \inf \big\{\norm{a}_{L_p(\mM)}: -a \leq x_n \leq a,\ a>0,\ \forall n\in\Z \big\}.
    \end{align*}

The concept of the weak maximal norm is also essential in this paper.
    \begin{definition}
	For $1\leq p< \infty$ and $x=(x_n)_{n\in\Z}\subset L_p(\mM)$, we define 	
	   \begin{align*}
	   \norm{x}_{\Lambda_{p,\infty}(\mM;\ell_{\infty})}
	   = \sup_{\lambda>0} \lambda \inf_{e\in\mP(\mM)} \big\{ \big( \tau(e^{\perp}) \big)^{\frac{1}{p}} : \norm{e x_n e}_{L_{\infty}(\mM)} \leq \lambda,\ \forall n\in\Z \big\},
	   \end{align*}
	where $\mP(\mM)$ is the set of all projections in $\mM$.
	A sequence $x = (x_n)_{n\in\Z}$ is said to belong to $\Lambda_{p,\infty}(\mM;\ell_{\infty})$ if $\norm{x}_{\Lambda_{p,\infty}(\mM;\ell_{\infty})} < \infty$.
    \end{definition}

The space $\Lambda_{p,\infty}(\mM;\ell_{\infty})$ is a quasi-Banach space.
For a positive sequence $x = (x_n) \subset L_p(\mM)$, the weak norm also admits the equivalent expression:
    \begin{align*}
	\norm{x}_{\Lambda_{p,\infty}(\mM;\ell_{\infty})}
	=\sup_{\lambda>0} \lambda \inf_{e\in\mP(\mM)} \Big\{ \big(\tau(e^{\perp}) \big)^{\frac{1}{p}} : 0\leq e x_n e \leq \lambda, \ \forall n\in\Z\Big\}.
    \end{align*}


\subsection{Noncommutative square functions}\label{Appendix: A Holder inequality in column spaces}
Consider a measure space $(\Omega, \mu)$ and an operator-valued function $f$.
Define
    \begin{align*}
    \norm{f}_{L_p(\mM;L_2^c(\Omega))} = \Big\|\Big(\int_{\Omega} f^*(t)f(t)\dif\mu (t) \Big)^{\frac{1}{2}}\Big\|_{L_p(\mM)}.
    \end{align*}
Then, the column space $L_p(\mM;L_2^c(\Omega))$ is defined as the space of all such functions $f$ for which $\norm{f}_{L_p(\mM;L_2^c(\Omega))} <\infty$.
When the measure is clear, we write simply $\dif\omega$ instead of $\dif\mu(\omega)$.

First, we introduce a H\"older inequality for the square functions.
    \begin{lemma} \label{a useful lemma for bad function}
    Let $0<p,q,r\leq\infty$ be such that $\frac{1}{r} = \frac{1}{p} + \frac{1}{q}$.
    Then for any $f\in L_p(\mM; L_2^c(\Sigma))$ and $g\in L_q(\mM; L_2^c(\Sigma))$,
    \begin{align*}
	\Big\| \int_{\Sigma} f^*(\omega) g(\omega)\dif\omega\Big\|_{L_r(\mM)}
	\leq \Big\| \Big(\int_{\Sigma}|f(\omega)|^2 \dif\omega\Big)^{\frac{1}{2}}\Big\|_{L_p(\mM)}
	     \Big\|\Big(\int_{\Sigma}|g(\omega)|^2 \dif\omega\Big)^{\frac{1}{2}} \Big\|_{L_q(\M)}.
	\end{align*}
    \end{lemma}

The next result was communicated to the authors of \cite{hong2023maximal} by 
Cadilhac during their work on the maximal Calder\'on-Zygmund operators. We include a proof for completeness.
    \begin{lemma} \label{also a useful lemma for bad function}
     Let $0<p,q\leq\infty$.
     For any $f\in L_p(\mM; L_2^c(\Sigma))$ and $g\in L_q(\mM; L_2^c(\Sigma))$, there exists a contraction operator $m$ such that
    \begin{align*}
    	\int_{\Sigma} f^*(\omega) g(\omega)\dif\omega = \Big( \int_{\Sigma} f^*(\omega) f(\omega)\dif\omega \Big)^{\frac{1}{2}} m \Big( \int_{\Sigma} g^*(\omega) g(\omega)\dif\omega \Big)^{\frac{1}{2}}.
    \end{align*}
    \end{lemma}

\begin{proof}
Fix $\epsilon>0$ and set
    \begin{align*}
    F_{\epsilon}:= \int_{\Sigma} f^*(\omega)f(\omega) \dif \omega + \epsilon\Bone_{\mM}	
    \quad\text{and}\quad
    G_{\epsilon}:= \int_{\Sigma} g^*(\omega)g(\omega) \dif \omega + \epsilon\Bone_{\mM}.
    \end{align*}
Then, we define $m_{\epsilon}$ by
    \begin{align*}
    m_{\epsilon}:= F_{\epsilon}^{-\frac{1}{2}} \Big( \int_{\Sigma} f^*(\omega) g(\omega) \dif \omega\Big) G_{\epsilon}^{-\frac{1}{2}},
    \end{align*}
so that
    \begin{align}\label{pojpkljl}
    \int_{\Sigma} f^*(\omega) g(\omega) \dif \omega
    = F_{\epsilon}^{\frac{1}{2}} m_{\epsilon} G_{\epsilon}^{\frac{1}{2}}.
    \end{align}
We first show that each $m_{\epsilon}$ is a contraction in $\mM$.
To prove this, we rewrite $m_{\epsilon}$ as
    \begin{align*}
    m_{\epsilon}
    = \int_{\Sigma} \big(f(\omega) F_{\epsilon}^{-\frac{1}{2}}\big)^* \big(g(\omega) G_{\epsilon}^{-\frac{1}{2}}\big) \dif \omega.
    \end{align*}
Applying Lemma \ref{a useful lemma for bad function}, we get
    \begin{align*}
    \| m_{\epsilon}\|_{L_{\infty}(\mM)}
    &\leq \Big\| F_{\epsilon}^{-\frac{1}{2}} \Big( \int_{\Sigma} |f(\omega)|^2 \dif\omega \Big)^{\frac{1}{2}} F_{\epsilon}^{-\frac{1}{2}}\Big\|_{L_{\infty}(\mM)}   \Big\| G_{\epsilon}^{-\frac{1}{2}} \Big( \int_{\Sigma} |g(\omega)|^2 \dif\omega \Big)^{\frac{1}{2}} G_{\epsilon}^{-\frac{1}{2}}\Big\|_{L_{\infty}(\mM)} \\
    &= \big\| F_{\epsilon}^{-\frac{1}{2}} F_0 F_{\epsilon}^{-\frac{1}{2}}\big\|_{L_{\infty}(\mM)}   \big\| G_{\epsilon}^{-\frac{1}{2}} G_0 G_{\epsilon}^{-\frac{1}{2}}\big\|_{L_{\infty}(\mM)} \\
    &\leq 1,
    \end{align*}
where the last inequality follows from $F_0\leq F_{\epsilon}$ and $G_0\leq G_{\epsilon}$.
Then, the $w^*$-compactness of the unit ball in $\mM$ implies that we can find an operator $m$ such that $\| m\|_{L_{\infty}(\mM)} \leq 1$ and $m_{\epsilon} \overset{w^*}{\rightarrow} m$ as $\epsilon\rightarrow 0$.
Moreover, we see that $F_{\epsilon} \overset{w^*}{\rightarrow} F_0$, $G_{\epsilon} \overset{w^*}{\rightarrow} G_0$ as $\epsilon\rightarrow 0$.
Passing to the limit in (\ref{pojpkljl}), we obtain    \begin{align*}
    \int_{\Sigma} f^*(\omega) g(\omega) \dif \omega
    = F_{0}^{\frac{1}{2}} m G_{0}^{\frac{1}{2}},
    \end{align*}
which is the desired factorization.
\end{proof}

We refer the reader to \cite{PisierQuanhua2003} for more information about Hilbert valued operator spaces.

\begin{acknowledgements}
We thank L. Cadilhac for his comment about the $L_q$-mean Dini condition,  which helped us improve this paper.
\end{acknowledgements}

\vskip1cm

\bibliographystyle{amsplain}
\bibliography{rf}

\end{document}